\documentclass[letterpaper, 10 pt, conference]{ieeeconf}  

\IEEEoverridecommandlockouts                              
\def\ARXIV{1}

\usepackage[style=ieee,backend=bibtex]{biblatex}
\usepackage{amsmath}
\usepackage{amssymb}
\usepackage{mathabx} 
\usepackage[scr=esstix]{mathalpha}
\usepackage{upgreek}
\usepackage{bm}
\usepackage{hyperref}
\hypersetup{colorlinks,
	linkcolor=blue,
	citecolor=blue,
	urlcolor=blue,
	linktocpage,
	plainpages=false}
\usepackage{cancel}
\usepackage{graphicx}
\usepackage{mathtools}
\usepackage{float}
\usepackage{tikz}
\usetikzlibrary{shapes,matrix,arrows,calc,positioning,fit,bayesnet,angles,patterns}

\tikzset{%
  every neuron/.style={
    circle,
    draw,
    minimum size=0.5cm
  },
  neuron missing/.style={
    draw=none, 
    scale=2,
    text height=0.333cm,
    execute at begin node=\color{black}$\vdots$
  },
}

\newtheorem{theorem}{Theorem}

\newtheorem{lemma}{Lemma}

\newtheorem{remark}{Remark}
\newtheorem{assumption}{Assumption}

\newcommand{\norm}[1]{\left\|#1\right\|}
\newcommand{\normB}[1]{\left\|#1\right\|_\Bc}
\newcommand{\normK}[1]{\left\|#1\,\right\|_\Kc}

\newcommand{\abs}[1]{\left|#1\right|}

\newcommand*{\medcap}{\mathbin{\scalebox{1.5}{\ensuremath{\cap}}}} 

\newcommand{\E}{\mathbb{E}}
\newcommand{\N}{\mathbb{N}}
\renewcommand{\P}{\mathbb{P}}

\newcommand{\R}{\mathbb{R}}

\DeclareMathOperator{\Exp}{\E}

\newcommand{\fb}{\bm{f}}

\newcommand{\vb}{\bm{v}}

\newcommand{\upgammab}{{\bm{\upgamma}}}

\newcommand{\Thetah}{{\widehat{\Theta}}}

\newcommand{\Bc}{\mathcal{B}}

\newcommand{\Dc}{\mathcal{D}}
\newcommand{\Fc}{\mathcal{F}}
\newcommand{\Hc}{\mathcal{H}}
\newcommand{\Ic}{\mathcal{I}}
\newcommand{\Lc}{\mathcal{L}}
\newcommand{\Kc}{\mathcal{K}}

\newcommand{\Rc}{\mathcal{R}}
\newcommand{\Sc}{\mathcal{S}}

\newcommand{\Vc}{\mathcal{V}}

\newcommand{\as}{\mathscr{a}}
\newcommand{\bs}{\mathscr{b}}

\newcommand{\dscr}{\mathscr{d}}
\newcommand{\fs}{\mathscr{f}}
\newcommand{\gs}{\mathscr{g}}

\newcommand{\ps}{\mathscr{P}}

\newcommand{\Bs}{\mathscr{B}}
\newcommand{\Cs}{\mathscr{C}}
\newcommand{\Ds}{\mathscr{D}}
\newcommand{\Fs}{\mathscr{F}}

\newcommand{\fw}{\widetilde{f}}

\newcommand{\uw}{\widetilde{u}}

\newcommand{\Lcw}{\widetilde{\Lc}}

\newcommand{\upgammaw}{\widetilde{\upgamma}}

\newcommand{\xh}{\widehat{x}}

\renewcommand{\d}{\text{d}} 
\newcommand{\dx}{\text{d}x} 

\newcommand{\Ind}{\mathbb{I}}

\newcommand{\sstep}{\sigma_{c\Sc}}

\newcommand{\relu}{\sigma_\Rc}

\newcommand{\fch}{\widecheck{f}}

\newcommand{\sigmach}{\widecheck{\sigma}}

\newcommand{\fh}{\widehat{f}}

\newcommand{\fst}{f^\star}
\newcommand{\fbN}{\bar{f}_N}
\newcommand{\fbbN}{\bar{\fb}_N}

\newcommand{\fhN}{\widehat{f}_N}
\newcommand{\fhR}{\widehat{\Fs}_R}
\newcommand{\fhlN}{\widehat{f}_{\lambda,N}}

\newcommand{\ftr}{\widetilde{f}_R}

\newcommand{\deltano}{\delta^{n+1}}
\newcommand{\deltal}{\hat{\delta}_\lambda}
\newcommand{\deltalno}{\hat{\delta}_\lambda^{n+1}}

\newcommand{\EB}{E_\Bc}

\newcommand{\pmin}{P_{\lambda\!\min}}

\newcommand{\upmax}{\upgamma_{\!\lambda\!\max}}

\newcommand{\Upsilonl}{{\Upsilon_{\!\lambda}}}

\newcommand{\Upsilonil}{{\Upsilon_{i,\lambda}}}
\newcommand{\Dt}{D_{\Upsilon\!\lambda}}
\newcommand{\gl}{\gs_\lambda}

\newcommand{\Dsig}{\Ds_\sigma}
\newcommand{\Lsig}{\Lc_\sigma}
\newcommand{\Pl}{P_\lambda}

\newcommand{\parX}{\upgamma^\top\!X}

\newcommand{\pariX}{\upgamma_i^\top\!X}
\newcommand{\pariY}{\upgamma_i^\top\!Y}
\newcommand{\parjX}{\upgamma_j^\top\!X}
\newcommand{\parbjX}{\upgammab_j^\top\!X}

\newcommand{\Cgl}{\gs_{\lambda\!\max}}

\newcommand{\KD}{K_{\Dsig}\!}
\newcommand{\KL}{K_{\Lsig}\!}

\newcommand{\epsbase}{\varepsilon_{\!N}}

\newcommand{\muB}{\bar\mu_\Bc}

\newcommand{\mubKc}{\bar\mu_\Kc}
\newcommand{\mubKcdx}{\bar\mu_\Kc(\dx)}

\newcommand{\muBdx}{\bar\mu_\Bc(\dx)}

\begin{document}

\title{\LARGE \bf
  Function Approximation with Randomly Initialized Neural Networks for Approximate Model 
  Reference Adaptive Control
}

\author{Tyler Lekang and Andrew Lamperski
      \thanks{T. Lekang and A. Lamperski are with the department of Electrical and
        Computer Engineering, University of Minnesota, Minneapolis,
        MN 55455, USA 
        {\tt\small lekang@umn.edu, alampers@umn.edu}}
}

\maketitle
\thispagestyle{empty}
\pagestyle{empty}

\begin{abstract}
Classical results in neural network approximation theory show how arbitrary continuous functions 
can be approximated by networks 
with a single hidden layer, under mild assumptions on the activation function. However, the 
classical theory does not give a constructive means to generate the network parameters that 
achieve a desired accuracy. Recent results have demonstrated that for specialized activation 
functions, such as ReLUs, high accuracy can be achieved via linear combinations of 
\textit{randomly initialized} activations. 
These recent works utilize specialized integral representations of target functions that depend 
on the specific activation functions used. 
This paper defines \textit{mollified integral representations}, which provide a means to form 
integral representations of target functions using activations for which no direct integral 
representation is currently known.
The new construction enables approximation guarantees for randomly initialized networks using 
any activation for which there exists an established base approximation which may not be 
constructive.
We extend the results to the supremum norm and show how this enables application to an extended, 
approximate version of (linear) model reference adaptive control.
\end{abstract}

\section{Introduction}

Adaptive control lies at the intersection of machine learning methods and the control of 
dynamics systems \cite{gaudio2019connections}, and has considerable development in the controls 
literature \cite{slotine1991applied,sastry1989adaptive,khalil2017high}. Deep neural network 
based machine learning approaches have achieved unprecedented levels of performance and utility
in areas such as language processing \cite{otter2020survey}, medical computer vision
\cite{esteva2021deep}, and reinforcement learning \cite{degrave2022magnetic}. In this paper, we
aim to contribute novel ideas to function approximation techniques and apply them to an 
approximate extension of Model Reference Adaptive Control (MRAC) established in 
\cite{lavretsky2013robust}.

A key challenge in function approximation work is that many theorems describe the existence of 
approximations that are linear combinations of some activation function (such as a sigmoid), but 
are not constructive. In \cite{pinkus1999approximation}, many of these results are reviewed, 
including classic results on the universal approximation properties of continuous 
sigmoidal activations 
\cite{cybenko1989approximation,hornik1989multilayer,funahashi1989approximate}, as well as 
another classic result by \cite{barron1993universal} for arbitrary sigmoidal
activations, which was expanded to arbitrary hinge functions by \cite{breiman1993hinging}.

Random approximation provides a key way to avoid the non-constructiveness by randomly sampling 
the internal parameters of the activations from a known set, if the target function of interest 
has an integral representation. However, determining these representations for general classes 
of target functions and activations is a nontrivial task.
\cite{irie1988capabilities} gives a constructive
method, but only for target and activation functions in $L_1$. In \cite{kainen2010integral} and 
\cite{petrosyan2020neural}, they propose constructive
methods for a class of target functions with unit step and ReLU activations respectively. In 
\cite{hsu2021approximation}, functions are approximated using trigonometric polynomial ridge 
functions, which can then be shown in expectation to be equivalent to randomly initialized ReLU 
activations.

There are several interesting and important results in the literature having to do with
neural networks with random (typically Gaussian) parameter initialization, sometimes as a 
consequence of using randomly initialized gradient descent for training the network. A classic
result by \cite{neal1994priors} shows that the output of a single hidden-layer network with
Gaussian randomly initialized parameters goes to a Gaussian Process as the width goes to
infinity. Similar results are then achieved in \cite{lee2017deep} for deep fully-connected
networks as all hidden layer widths go to infinity. Also for deep fully-connected networks, in
\cite{jacot2018neural} the authors define the Neural Tangent Kernel and propose that its limit,
as the hidden layer widths go to infinity, can be used to study the timestep evolution and
dynamics of the parameters, and the corresponding network output function, in gradient descent.
In \cite{poole2016exponential}, the authors show that single hidden-layer networks cannot achieve
the same rates of increase in a measure of curvature produced by the network output, as deep
networks can, with parameters Gaussian randomly initialized and bounded activation functions. In
\cite{arora2019fine}, the authors show that single hidden-layer networks of a sufficient width
can use Gaussian randomly initialized gradient descent on values of a target function, and
achieve guaranteed generalization to the entire function. And in \cite{du2019gradient}, the
authors show that deep fully-connected networks, where each hidden layer width meets a
sufficient size, can be trained with Gaussian randomly initialized gradient descent and be
guaranteed to reach the global minimum at a linear rate.

Our primary contributions in this paper are: developing a novel method for bridging convex 
combinations of activation functions (eg. step and ReLU) with unknown parameters to
approximations using randomly initialized parameters by using a mollified integral 
representation, obtaining a main overall result bounding the approximation error for the random 
approximations in the $L_2$ norm, extending these results to the $L_\infty$ norm under certain 
conditions, and then applying those results to an approximate extension of MRAC adaptive control.

The organization of the remaining parts of the paper are as follows. Section~\ref{sec:notation}
presents preliminary notation. Section~\ref{sec:apx:background} presents background on function 
approximation with neural networks. Section~\ref{sec:apx:mollApx} presents theoretical results 
on 
mollified approximations, while Section~\ref{sec:apx:random} gives an overall approximation 
error 
bound in $L_2$ norm for randomly initialized approximations. Section~\ref{sec:apx:application}
presents an extension to $L_\infty$ norm bounds and an application on approximate MRAC, and we 
provide closing remarks in Section~\ref{sec:conclusion}.

\section{Notation}
\label{sec:notation}

We use $\R,\N$ to denote the real and natural numbers. $\Ind$ denotes the indicator 
function, while $\Exp$ denotes the expected value. We use $\Bc,\Kc$ to denote bounded and 
compact convex subsets of $\R^n$. The integer set $\{1,\dots,k\}$ is denoted by $[k]$. 
$B_x(r)\subset\R^n$ denotes the radius $r$ Euclidean ball centered at $x\in\R^n$.
We interpret $w,x\in\R^n$ as column vectors and denote their inner product as $w^\top\!x$.
Similarly, we denote the product of matrix $W\in\R^{n\times N}$ with vector $x$ as $W^\top\!x$,
which is a length $N$ vector where the $i$th element is the inner product $W_i^\top\!x$.
Subscripts on vectors and matrices denote the row index, for example $W_i^\top$ is the $i$th row
of $W^\top$. The standard Euclidean norm is denoted
$\|w\|$.
We use subscripts on
time-dependent variables to reduce parentheses, for example $\phi(x(t))$ is instead denoted
$\phi(x_t)$. The $i$th row of a time-varying vector or matrix is then $x_{t,i}$.

\section{Background on Function Approximation with Neural Networks}
\label{sec:apx:background}
In this paper, we consider the approximation of target functions $f:\R^n\to\R$ that are
continuous on a bounded set $\Bc\subset\R^n$ containing the origin, which we denote as $f\in 
\Cs(\Bc)$.

Formally, for a given target function $f\in \Cs(\Bc)$, positive integer $N\in\N$, and positive 
scalar $\epsbase>0$ (which may depend on $N$) there exists some vector $\theta\in\R^N$ and 
nonlinear, continuous almost everywhere basis functions $\Bs_1,\dots,\Bs_N:\R^n\to\R$ such that 
the approximation
\begin{equation}\label{eq:apx:genApx}
\fh(x) \ := \ \sum_{i=1}^N \theta_i\,\Bs_i(x)
\end{equation}
satisfies
\begin{equation}\label{eq:apx:normBound}
\normB{\fch(x)-\fh(x)} \ \leq \ \epsbase \ \ .
\end{equation}
Here, we denote the the $L_2(\Bc,\muB)$ norm
$$
\normB{\fch(x)} \ := \ \left(\int_B|\fch(x)|^2\,\muBdx\right)^{\frac{1}{2}}
$$
with $\muB$ as the uniform probability measure over $\Bc$ such that $\int_\Bc\muBdx=1$, and some
\textit{affine shift} $\fch(x)=f(x)-\mathscr{a}^\top\!x -\mathscr{b}$, for some 
$\mathscr{a}\in\R^n$ and $\mathscr{b}\in\R$.

We will assume that the basis function is parameterized $\Bs(x\,;\upgamma)$ and is a 
composition of some nonlinear, continuous almost everywhere function $\sigma:\R\to\R$ together
with an affine transformation of $\R^n$ into $\R$ according to the \textit{parameter} 
$\upgamma\in\Upsilon$. We define each point $\upgamma:=[w^\top\ b]^\top=[w_1\ \dots\ w_n\ 
b]^\top$ with $w\in\R^n\setminus\{0_n\}$ as the \textit{weight vector} and $b\in\R$ as the
\textit{bias term}. We assume the \textit{parameter set} $\Upsilon$ is a bounded subset of 
$\R^n\setminus\{0_n\}\times\R$. And so, we set $\Bs(x\,;\upgamma_i)=\sigma(w_i^\top 
x+b_i)=\sigma(\pariX)$, where 
$X:=[x_1\ \cdots\ x_n\ 1]^\top$ and we denote $\upgamma_i=[w_i^\top\ b_i]^\top=[w_{i,1}\
\dots\ w_{i,n}\ b_i]^\top$, for all $i\in[N]$. 

We will use the term \textit{base approximation} to define, for some vector $\theta\in\R^N$ and
\underline{fixed} parameters $\upgamma_1,\dots,\upgamma_N\in\Upsilon$, for any integer $N\in\N$,
the sum
\begin{equation}\label{eq:apx:baseApx}
\fhN(x) = \sum_{i=1}^{N} \theta_i\,\sigma(\pariX)
\ \ ,
\end{equation}
and \textit{random approximation} to define, for
some vector $\vartheta\in\R^R$ and \textit{randomly sampled} parameters
$\upgammab_1,\dots,\upgammab_R\in\Upsilon$, for any integer $R\in\N$, the sum
\begin{equation}\label{eq:apx:rndApx}
\fhR(x\,;\upgammab_1,\dots,\upgammab_R) = \sum_{j=1}^{R}\vartheta_j\,\sigma(\parbjX)
\ \ .
\end{equation}

Note that these linear combinations are equivalent to a feedforward 
neural network with a single hidden layer, as depicted in Figure~\ref{fig:apx:SHLNN}, wherw 
$x_0=1$. Hence why we refer to $\sigma$ as an \textit{activation function}, as the output of 
each neuron is the activation of the weighted inputs (plus bias term) to the neuron. We will 
assume that the activation function $\sigma$ is bounded (growth)
as follows.

\begin{figure}
\centering
\begin{tikzpicture}[x=1.2cm, y=0.75cm, >=stealth]

\foreach \m/\l [count=\y] in {1,2,missing,3}
  \node [every neuron/.try, neuron \m/.try] (input-\m) at (0,1.5-\y) {};

\foreach \m [count=\y] in {1,missing,2}
  \node [every neuron/.try, neuron \m/.try ] (hidden-\m) at (2,2.5-\y*1.75) {};

\foreach \m [count=\y] in {1}
  \node [every neuron/.try, neuron \m/.try ] (output-\m) at (4,0.1-\y) {};


\foreach \l [count=\i] in {0,1,n}
  \draw [<-] (input-\i) -- ++(-1,0)
    node [above, midway] {$x_\l$};

\foreach \l [count=\i] in {1,n}
  \node [above] at (hidden-\i.north) {$\sigma$};

\foreach \l [count=\i] in {1}
  \draw [->] (output-\i) -- ++(1,0)
    node [above, midway] {$\fh(x)$};

\foreach \i in {1,...,3}
  \foreach \j in {1,...,2}
    \draw [->] (input-\i) -- (hidden-\j);

\foreach \i in {1,...,2}
  \foreach \j in {1}
    \draw [->] (hidden-\i) -- (output-\j);


\node[] at (0.9,0.9) {$b_1$};
\node[] at (1.5,-1.4) {$b_N$};
\node[] at (0.8,0.3) {$w_{1,1}$};
\node[] at (1.0,-0.3) {$w_{1,n}$};
\node[] at (0.9,-2.0) {$w_{N,1}$};
\node[] at (1,-2.9) {$w_{N,n}$};
\node[] at (3,0.2) {$\theta_1$};
\node[] at (3,-2.3) {$\theta_N$};

\end{tikzpicture}
\caption{\label{fig:apx:SHLNN} Single hidden-layer neural network with $\sigma$ activation 
function.}
\end{figure}
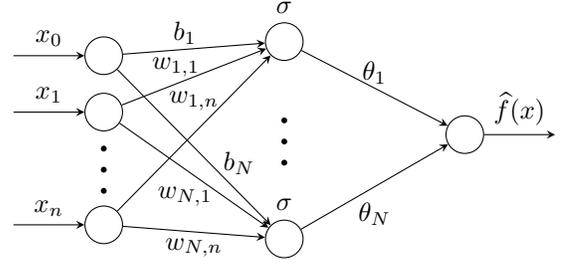

\begin{assumption}\label{sigAsmpt}
\textit{
The nonlinear, continuous almost everywhere activation function $\sigma:\R\to\R$ is bounded 
(growth) over all feasible inputs, such that either of}
\begin{enumerate}
\item $|\sigma(u) - \sigma(v)| \ \leq \ \Dsig$
\item $|\sigma(u) - \sigma(v)| \ \leq \ \Lsig\,|u-v|$
\end{enumerate}
\textit{
hold for all $u,v\in\Ic:=\{\parX \ | \ (\upgamma,X)\in\Upsilon\times\Bc\times\{1\}\}$,
with positive scalar $\Dsig$ or $\Lsig$.}
\end{assumption}

If we shift the activation to the origin as $\sigmach(y):=\sigma(y)-\sigma(0)$, this maintains 
the assumption with the same constant in either case.
The (scaled) step function $\sstep$ and the ReLU function $\relu$, defined as
\begin{align}
\sstep(y) &= 
\begin{cases*}\label{eq:pe:step}
0 \quad\text{if } y\leq0 \\
c \quad\text{if } y>0
\end{cases*}
\qquad \text{(scaled step)} \\
\relu(y) &= 
\begin{cases*}\label{eq:pe:ReLU}
0 \quad\text{if } y\leq0 \\
y \quad\text{if } y>0
\end{cases*}
\qquad \text{(ReLU)} \ \ ,
\end{align}
are examples meeting Assumption~\ref{sigAsmpt}, with $\Dsig=c$ and $\Lsig=1$ respectively,
where $c>0$ is an arbitrary positive scalar.

\section{Mollified Approximations}
\label{sec:apx:mollApx}
In this section, we define the \textit{mollified approximation} $\fhlN$ using a novel method of 
forming a \textit{mollified integral representation} for a base approximation $\fhN$ to a 
target function $f\in\Cs(\Bc)$, and bound their difference in the $L_2(\Bc,\muB)$ norm.

\subsection{Functions with Integral Representation}
If a target function $f:\R^n\to\R$ has an \textit{integral representation} over the parameters
set $\Upsilon$ with some activation function $\sigma:\R\to\R$, then
\begin{equation}\label{eq:apx:intRep}
f(x) \ = \ \int_\Upsilon \gs(\upgamma)\,\sigma(\parX)\,\d\upgamma
\end{equation}
holds with some bounded, continuous \textit{coefficient function} \mbox{$\gs:\Upsilon\to\R$}. 
This is a version of \eqref{eq:apx:baseApx} that is continuous in the parameters $\upgamma_i$ 
over $\Upsilon$. Indeed, if we define $\upgamma_1,\dots,\upgamma_N$ to be a sequence of points 
that uniformly grids $\Upsilon$ with volume $\d_N$ between any neighboring points, then as
$N\to\infty$ (and $\d_N\to0$) we have
$$
\lim\limits_{N\to\infty}\sum_{i=1}^{N}\gs(\upgamma_i)\,\sigma(\pariX)\,\d_N 
= \int_\Upsilon \gs(\upgamma)\,\sigma(\parX)\,\d\upgamma = f(x)
\ .
$$
This implies that the (infinite) sum of the absolute values of the coefficients is finite, with
$$
\lim\limits_{N\to\infty}\sum_{i=1}^{N}|\gs(\upgamma_i)|\,\d_N \ 
= \ \int_\Upsilon |\gs(\upgamma)|\,\d\upgamma \ < \ \infty
$$
holding since $\gs$ is bounded and the integral is over the bounded set $\Upsilon$.

However, finding such integral representations \eqref{eq:apx:intRep} for general classes of
functions $f$ and activation functions $\sigma$ is a nontrivial task. If \textit{both}
are $L_1$ (absolutely integrable over $\R^n$ and $\R$ respectively), then Theorem 1 in
\cite{irie1988capabilities} shows that the integral representation can always be constructed
with $\gs$ related to the Fourier transforms of $f$ and $\sigma$. But this is a limiting
restriction, since none of the typical activation functions like sigmoids, steps, ReLU's, and 
their variants meet this requirement. Even so, in \cite{funahashi1989approximate} this was used 
to prove the universal approximation properties of sigmoid activations.

\subsection{Mollified Integral Representation}
\label{sec:apx:moll}
Now we will propose a novel method of forming an integral representation for
\textit{any} function $f\in\Cs(\Bc)$, so long as a base approximation $\fhN$ of the form 
\eqref{eq:apx:baseApx} exists and the bounded parameter set $\Upsilon\subset\R^{n+1}$ and a
bound $\Sc_N$ on the coefficient size $\sum_{i=1}^N|\theta_i|$ are known.

For any fixed parameters $\upgamma_1,\dots,\upgamma_N\in\Upsilon$, this approximation $\fhN$ 
can be defined as
\begin{align}\nonumber
\fhN(x) \ =& \  \sum_{i=1}^{N} \theta_i\,\sigma(\pariX) \\\label{eq:apx:baseIntRep}
=& \ 
\int_\Upsilon \,\sum_{i=1}^N\theta_i\,\deltano(\upgamma-\upgamma_i)\ \sigma(\parX)\,\d\upgamma \ 
\ ,
\end{align}
where $\deltano$ is the $(n+1)$-dimensional Dirac delta which satisfies for any
$\upgamma_i\in\Upsilon$ that
$\int_\Upsilon\deltano(\upgamma-\upgamma_i)\,\sigma(\parX)\,\d\upgamma =\sigma(\pariX)\ $.

Thus, the RHS of \eqref{eq:apx:baseIntRep} is like an integral representation 
\eqref{eq:apx:intRep} of $\fhN$ with a coefficient mapping \mbox{$\gs(\upgamma) = \sum_{i=1}^N 
\theta_i\,\deltano(\upgamma-\upgamma_i)$}.
However, this mapping is not continuous and bounded as required, and the Dirac delta is not even
a function in the regular sense. Indeed, if the overall coefficients of a random approximation 
$\fhR$ were set as $\vartheta_j=\gs(\upgammab_j)$ for each randomly sampled 
$\upgammab_1,\dots,\upgammab_R\in\Upsilon$, then in this case they almost surely would all be 
zero.

We propose the \textit{Mollified integral representation} to solve this issue, where the Dirac
delta in \eqref{eq:apx:baseIntRep} is replaced with the \textit{mollified} delta function
$\deltalno$. Recalling that $\upgamma=[w_1\ \cdots\ w_n\ b\,]$, this is defined as the 
$(n+1)$-dimensional bump function
\begin{align}\label{eq:apx:mollDel}
&\deltalno(\upgamma) \ := \ \\\nonumber
&\begin{cases}
(\eta\lambda)^{n+1}\prod_{i=1}^n
\exp\left(\dfrac{-1}{1-\lambda^2w_i^2}\right)
\exp\left(\dfrac{-1}{1-\lambda^2b^2}\right)\\[8pt]
\hspace{160pt} \upgamma\in(-\frac{1}{\lambda},\frac{1}{\lambda})^{n+1} \\
0 \hspace{155pt} \text{otherwise}
\end{cases}
\end{align}
with $\eta^{-1}=\int_{-1}^{1}\exp\left(\frac{-1}{1-y^2}\right)\d y \approx0.444$ thus giving
$\int\deltalno(\upgamma-\upgamma_i)\,\d\upgamma = \int\deltano(\upgamma-\upgamma_i)\,\d\upgamma
= 1$ for any $n\geq1$ and $\upgamma_i\in\R^{n+1}$.

We term the positive scalar $\lambda>0$ as the \textit{mollification factor}. This controls the
height and width of the \mbox{mollified} delta $\deltalno$. Since the height scales as
$\lambda^{n+1}$ and the integral remains constant, this requires the width to shrink. 
Figure~\ref{fig:apx:mollFactor} plots the 1-dimensional case for different $\lambda$ values.
\begin{figure}
\centering
\includegraphics[width=\columnwidth]{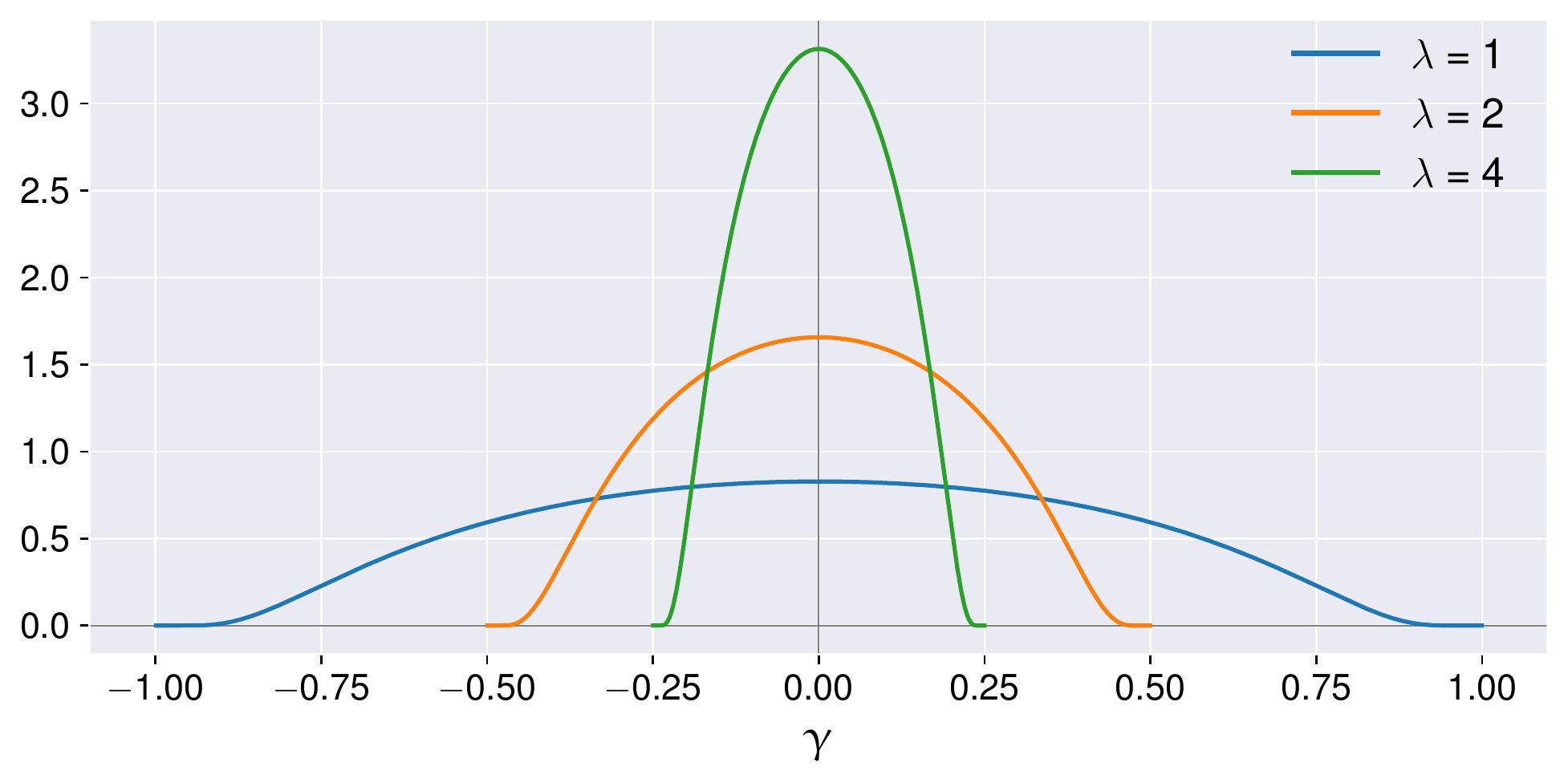}
\caption{\label{fig:apx:mollFactor}The mollified delta $\deltal$ for different $\lambda$ values.
In each case, $\int_{-1/\lambda}^{1/\lambda}\deltal(\upgamma)\d\upgamma=1$.}
\end{figure}

The support of $\deltalno(\upgamma-\upgamma_i)$ is the $(n+1)$-dimensional open cube set
$(-\frac{1}{\lambda},\frac{1}{\lambda})^{n+1}$ centered about any $\upgamma_i\in\Upsilon$. And
so, if $\upgamma_i$ is sufficiently close to (or on) the boundary of $\Upsilon$, then it no
longer holds that the entire support of $\deltalno(\upgamma-\upgamma_i)$ is within $\Upsilon$.
Thus, we define the \textit{$\lambda$-expanded} parameters set $\Upsilonl$ such that the support
is always fully contained, up to and include the boundary of $\Upsilon$. For example, if
$\Upsilon$ is rectangular product of intervals $[\alpha_1,\beta_1]\times\cdots\times
[\alpha_n,\beta_n]$, then each interval is expanded as 
$[\alpha_i-\frac{1}{\lambda},\beta_i+\frac{1}{\lambda}]$ to obtain $\Upsilonl$.
With this $\lambda$-expanded parameters set, we always maintain that
$\int_\Upsilonl\deltalno(\upgamma-\upgamma_i)\d\upgamma=1$ for all $\lambda>0$, $n\geq1$, and any
$\upgamma_i\in\Upsilon$.

Since $\Upsilon\subset\Upsilonl$ for all $\lambda>0$, then \eqref{eq:apx:baseIntRep} is 
equivalently
\begin{equation}\label{eq:apx:baseIntRepLam}
\fhN(x) = \int_\Upsilonl \,\sum_{i=1}^N\theta_i\,\deltano(\upgamma-\upgamma_i)\ 
\sigma(\parX)\,\d\upgamma \ \ .
\end{equation}
And so, by replacing the Dirac delta in \eqref{eq:apx:baseIntRepLam} with the mollified delta
$\deltalno$, for any choice of $\lambda>0$, we obtain the corresponding \textit{mollified 
approximation}
\begin{equation}\label{eq:apx:mollApx}
\fhlN(x) := \int_\Upsilonl \,\sum_{i=1}^N\theta_i\,\deltalno(\upgamma-\upgamma_i)\ 
\sigma(\parX)\,\d\upgamma
\end{equation}
using a \textit{mollified integral representation}.
This now gives the required continuous, bounded coefficient mapping for any fixed
$\upgamma_1,\dots,\upgamma_N\in\Upsilon$ as
\begin{equation}\label{eq:apx:gLam}
\gl(\upgamma) := \sum_{i=1}^N\theta_i\,\deltalno(\upgamma-\upgamma_i)
\ \ .
\end{equation}
We illustrate the transformation for the 1-dimensional case conceptually in 
Figure~\ref{fig:apx:moll}.
\begin{figure}
\centering
\includegraphics[width=.9\columnwidth]{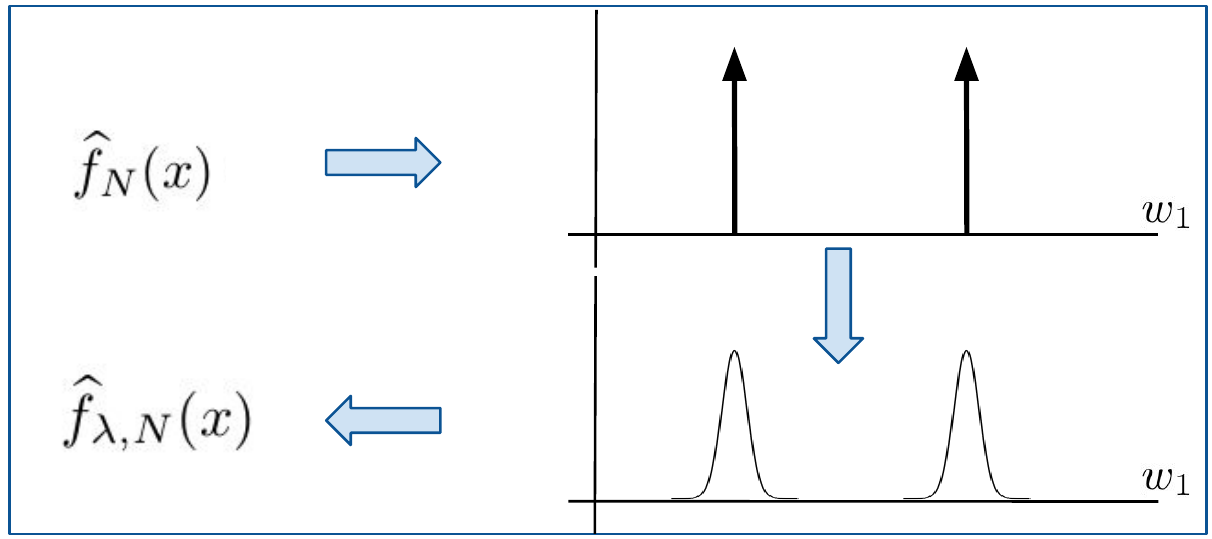}
\caption{\label{fig:apx:moll}The base approximation $\fhN$ is transformed into the mollified
approximation $\fhlN$ by replacing Dirac deltas $\delta$ with the mollified deltas $\deltal$.}
\end{figure}

\subsection{Main Result}
Let us define $E(x):=\|\,[x_1\ \cdots\ x_n\ 1]\,\|_2=\|X\|_2$ as the Euclidean norm mapping from 
$\R^n\times\{1\}$ to $\R$, and then define
\begin{equation}\label{eq:apx:EB}
\EB := \normB{E(x)}
\end{equation}
as the $L_2(\Bc,\muB)$ norm of this mapping over the Euclidean distances of $x\in\Bc$. We can now
state the main result, which bounds the difference between a base approximation and its 
corresponding mollified approximation, for a chosen mollifcation factor $\lambda$, in the 
$L_2(\Bc,\muB)$ norm.

\begin{theorem}\label{thm:apx:moll}
\textit{
Let there exist a base approximation $\fhN$ of the form \eqref{eq:apx:baseApx} using any 
activation function $\sigma$ satisfying Assumption~\ref{sigAsmpt}, for any $N\in\N$, with some 
unknown parameters $\upgamma_1,\dots,\upgamma_N\in\Upsilon$ and coefficients $\theta\in\R^N$, 
but having a known bound on the coefficient size of $\sum_{i=1}^N|\theta_i|\leq \Sc_N$ and the 
parameters set $\Upsilon$ known. Then, the mollified approximation $\fhlN$ over the 
$\lambda$-expanded parameters set $\Upsilonl$ in \eqref{eq:apx:mollApx} satisfies either of:}
\begin{align}\label{eq:apx:epslSupBound}
i)\hspace{55.5pt} \normB{\fhN-\fhlN} \ &\leq \ \Sc_N \, \Dsig \\\label{eq:apx:epslL2Bound}
ii)\hspace{52pt} \normB{\fhN-\fhlN} \ &\leq \ \Sc_N\, \frac{\Lsig\,\sqrt{n+1}}{\lambda}\ \EB \ .
\end{align}
\end{theorem}
\begin{proof}
\if\ARXIV1
Given in Appendix~\ref{app:apx:mollProof}.
\fi
\if\ARXIV0
Given in Appendix I of \cite{lekang2023functionfull}.
\fi
\end{proof}

\begin{remark}
The results of Theorem~\ref{thm:apx:moll} show that if the activation function $\sigma$ is 
Lipschitz over its feasible inputs, then as $\lambda\to\infty$ we have $\fhlN(x)\to\fhN(x)$ in 
the sense of the $L_2(\Bc,\muB)$ norm.
\end{remark}

\section{Randomly Initialized Approximations}
\label{sec:apx:random}

In this section, we focus on random approximations $\fhR$ of the form \eqref{eq:apx:rndApx} and 
show that they exist (in high probability) with approximation error in the 
$L_2\left(\Bc,\muB\right)$ norm 
of the shape $O(\frac{1}{\sqrt{R}})$, down to the equivalent bound $\epsbase$ as in 
\eqref{eq:apx:normBound} of an existing base approximation $\fhN$ as in \eqref{eq:apx:baseApx}.

This is done by creating the corresponding mollified approximation $\fhlN$, which has a 
mollified integral representation and therefore can be approximated with randomly sampled 
parameters, which gives $\fhR$.

\begin{figure}[H]
\centering
\begin{tikzpicture}[node distance=1.5em,
        nodes={draw,fill=white},
        box/.style={align=left,inner sep=1ex},
        marrow/.style={single arrow,
              single arrow head extend=1mm,
              execute at begin node={\strut}}]
   \node (A) [box] {$\fhN$ };
   \node (B) [box, right=of A] {$\fhlN$};
   \node (C) [box, right=of B] {$\fhR$};
\draw [->] (A) edge (B) (B) edge (C);
\end{tikzpicture}
\end{figure}

\subsection{Expected Value of Random Approximations to Functions with Integral Representations}

An important property of functions that have integral representations with some activation 
$\sigma$ is that random approximations $\fhR$ of the form \eqref{eq:apx:rndApx} exist which are
exactly equal to the function in expected value over the randomly drawn parameters.

And so, the following lemma builds upon the mollified integral representation of the mollified
approximation $\fhlN$ to show that there always exists random approximations with randomly 
sampled parameters from the $\lambda$-expanded set $\Upsilonl$ satisfying this property.

\begin{lemma}\label{lem:apx:exp}
\textit{
Let the target function $f\in\Cs(\Bc)$ satisfy Theorem~\ref{thm:apx:moll} and thus have the
mollified approximation $\fhlN$ as in \eqref{eq:apx:intRep} with activation $\sigma$ for the 
selected mollification factor $\lambda>0$, which has the integral representation
\eqref{eq:apx:baseIntRepLam} over the $\lambda$-expanded parameters set $\Upsilonl$.
If the random parameters $\upgammab_1,\dots,\upgammab_R$ are sampled iid according
to any density $\Pl$ which is nonzero over $\Upsilonl$, then there exists a random approximation
$\fhR$ of the form \eqref{eq:apx:rndApx} using the activation $\sigma$ such that}
\begin{equation}\label{eq:apx:expL}
\Exp_{\upgammab_j\sim\Pl}\limits\left[\fhR(x\,;\upgammab_1,\dots,\upgammab_R)\right] \
= \ \fhlN(x)
\ .
\end{equation}
\end{lemma}
\begin{proof}
\if\ARXIV1
Given in Appendix~\ref{app:apx:exp}.
\fi
\if\ARXIV0
Given in Appendix II of \cite{lekang2023functionfull}.
\fi
\end{proof}
\textit{Proof Sketch:} Let us first make the following definitions. With $\gl(\upgamma)$ as in 
\eqref{eq:apx:gLam} for any $\lambda>0$, we define
\begin{align}
\label{eq:apx:Cgl}
&\max_{\upgamma\in\Upsilonl}\limits|\gl(\upgamma)| \ \leq \ 
\sum_{i=1}^N|\theta_i|\left(\frac{\eta\lambda}{e}\right)^{n+1} \ =:\ \Cgl \ .
\end{align}
For the corresponding $\Upsilonl$, we define the diameter $\Dt:=\sup\{\|\upgamma-\phi\|_2 \ | \ 
\upgamma,\phi\in\Upsilonl\}$ and largest point $\upmax:=\sup\{\|\upgamma\|_2 \ | \ 
\upgamma\in\Upsilonl\}$, and for the selected $\Pl$ we define the minimum  $\pmin := 
\inf\{\Pl(\upgamma) \ | \ \upgamma\in\Upsilonl\}$.

Then, note that the RHS of \eqref{eq:apx:mollApx} is equivalently
\begin{align*}
&\fhlN(x) \ = \\
&\int_\Upsilonl \frac{\gl(\upgamma)}{\Pl(\upgamma)}\,\sigma(\upgamma^\top z)
\,\Pl(\upgamma)\d\upgamma
= \Exp_{\upgammab\sim\Pl}\limits
\left[\frac{\gl(\upgammab)}{\Pl(\upgammab)}\,\sigma(\upgammab^\top z)\right]
\ .
\end{align*}

Thus, a random approximation $\fhR$ exists which satisfies \eqref{eq:apx:expL} by setting its 
coefficients, for all $j\in[R]$, as
\begin{equation*}\label{eq:appApx:cjOpt}
\vartheta_j = \frac{\gl(\upgammab_j)}{\Pl(\upgammab_j)\,R}
\ =\ \frac{\sum_{i=1}^N\theta_i\,\deltalno(\upgammab_j-\upgamma_i)}{\Pl(\upgammab_j)\,R}
\ \ .
\end{equation*}

\subsection{Random Approximations to Base Approximations}
\label{sec:apx:main}

We now state and prove an overall bound on the approximation error in $L_2(\Bc,\muB)$ norm, 
between a base approximation $\fhN$ of the form \eqref{eq:apx:baseApx} to a target function 
$f\in \Cs(\Bc)$ and a random approximation $\fhR(x\,;\upgammab_1,\dots,\upgammab_R)$ of the form 
\eqref{eq:apx:rndApx}.

\begin{theorem}\label{thm:apx:main}
\textit{
Let there exist a base approximation $\fhN$ of the form \eqref{eq:apx:baseApx}, for some integer 
$N\in\N$, to a target function $f\in \Cs(\Bc)$, with unknown parameters 
$\upgamma_1,\dots,\upgamma_N$ in a known bounded parameter set $\Upsilon\subset\R^{n+1}$, 
unknown coefficients $\theta\in\R^N$ with a known bound on the coefficient size
$\sum_{i=1}^N|\theta_i|\leq\Sc_N$, and using an activation function $\sigma$ satisfying
Assumption~\ref{sigAsmpt}. Assume that $\normB{\fch(x)-\fhN(x)}\leq\epsbase$ holds for some 
$\epsbase>0$ which may depend on $N$, for an affine shift $\fch(x)=f-\as^\top\!x-\bs f(0)$ with
some $\as\in\R^n$ and $\bs\in\R$.
The mollified approximation $\fhlN$ as in \eqref{eq:apx:mollApx}
satisfies the results of Theorem~\ref{thm:apx:moll} for any mollification factor $\lambda>0$.
Let the random parameters $\upgammab_1,\dots,\upgammab_R$ be sampled iid according to any
nonzero density $\Pl$ over the $\lambda$-expanded set $\Upsilonl$, which has diameter $\Dt$ and
farthest point $\upmax$, and let $\EB$ be as in \eqref{eq:apx:EB}. Then for any $\nu\in(0,1)$,
with probability greater than $1-\nu$ there exists a random approximation 
$\fhR(x\,;\upgammab_1,\dots,\upgammab_R)$ as in \eqref{eq:apx:rndApx} with
$|\vartheta_j|\leq\frac{\Cgl}{\pmin R}$ for all $j\in[R]$ such that either of the following 
hold:}
\begin{align*}
&\normB{\,\fch(x) - \fhR(x\,;\upgammab_1,\dots,\upgammab_R)} \ \leq \\
&i)\hspace{34pt}
\epsbase + 
\Sc_N\Dsig\left(1 \ + \ 
\frac{\lambda^{n+1}}{\sqrt{R}\,}\,\KD(\nu)\right) \\
&ii)\hspace{30pt}
\epsbase + 
\Sc_N\Lsig\left(\frac{\sqrt{n+1}}{\lambda}\,\EB +
\frac{\lambda^{n+1}}{\sqrt{R}\,}\,\KL(\nu)\right) \ \  .
\end{align*}
The coefficients $\KD$ and $\KL$ are defined in terms of the selected $\nu$ as:
\begin{align*}
&\KD(\nu) \ := \ 
\frac{\eta^{n+1}}{e^{n+1}\pmin}
\left(
1 + \frac{|\sigma(0)|}{\Dsig} \right. \\
&\left. + \bigg(3 + 2\frac{|\sigma(0)|}{\Dsig}\bigg) 
\sqrt{\frac{1}{2}\log\left(\frac{1}{\nu}\right)}
\right) \\
&\KL(\nu) \ := \
\frac{\eta^{n+1}}{e^{n+1}\pmin}
\left(
\upmax\EB + \frac{|\sigma(0)|}{\Lsig}\right. \\
&\left. + \bigg((\Dt+2\upmax)\EB + 
2\frac{|\sigma(0)|}{\Lsig}\bigg) \sqrt{\frac{1}{2}\log\left(\frac{1}{\nu}\right)}
\right).
\end{align*}
\end{theorem}
\begin{proof}
\if\ARXIV1
Given in Appendix~\ref{app:apx:mainProof}.
\fi
\if\ARXIV0
Given in Appendix II of \cite{lekang2023functionfull}.
\fi
\end{proof}
\textit{Proof Sketch:} We closely follow the strategy in Lemma 4 of
\cite{rahimi2008weighted}, which uses McDiarmid's inequality to obtain the final result.

We start by defining function $h_\Bc:\Upsilon^R_\lambda\to\R$, which captures the norm difference
between the random approximation and its expected value over the iid draws of 
$(\upgammab_1,\dots,\upgammab_R)\in\Upsilon^R_\lambda$, as
\begin{align}
\label{eq:apx:hB}
&h_\Bc(\upgammab_1,\dots,\upgammab_R) \ := \\\nonumber
&\normB{\,\fhR(x\,;\upgammab_1,\dots,\upgammab_R)
-\Exp_{\upgammab_j\sim\Pl}\limits\left[\fhR(x\,;\upgammab_1,\dots,\upgammab_R)\right]} \ \ .
\end{align}
Thus, we can use Lemma~\eqref{lem:apx:exp} to show that there exists $\fhR$ such that 
\eqref{eq:apx:hB} is equivalently
$$
h_\Bc(\upgammab_1,\dots,\upgammab_R) \ = \ \normB{\,\fhlN(x) - 
\fhR(x\,;\upgammab_1,\dots,\upgammab_R)} \ .
$$
We then can bound the expected value
$$
\Exp_{\upgammab_j\sim\Pl}\limits\Big[h_\Bc(\upgammab_1,\dots,\upgammab_R)\Big]
$$
and the absolute value of the difference between any 
$$
\big|h_\Bc(\upgamma_1,\dots,\upgamma_j,\dots,\upgamma_R)-
h_\Bc(\upgamma_1,\dots,\widetilde\upgamma_j,\dots,\upgamma_R)\big|
$$
that differ only at one coordinate, over each $j\in[R]$. These two bounds then allow us to apply 
the results of McDiarmid's inequality, thus bounding (in high probability)
$$
\normB{\,\fhlN(x) - 
\fhR(x\,;\upgammab_1,\dots,\upgammab_R)} \ .
$$
The proof is then completed by using the triangle inequality, with
\begin{align*}
&\normB{\,\fch(x) - \fhR(x\,;\upgammab_1,\dots,\upgammab_R)} \ \leq \ 
\normB{\fch(x)-\fhN(x)} +\\ &\ \normB{\fhN(x)-\fhlN(x)} 
+\normB{\fhlN(x)-\fhR(x\,;\upgammab_1,\dots,\upgammab_R)} \ .
\end{align*}
where the first term is assumed bounded by $\epsbase$ for the base approximation $\fhN$ and the 
second term is bounded by Theorem~\ref{thm:apx:moll}.

\begin{remark}
Classic results from \cite{barron1993universal} and \cite{breiman1993hinging} prove the existence
of \textit{convex} combinations of step \eqref{eq:pe:step} and ReLU \eqref{eq:pe:ReLU} 
activations that approximate (affine shifts of) functions in specific classes arbitrarily well, 
with an 
approximation error proportional to $\epsbase=O\left(\frac{1}{\sqrt{N}}\right)$. Importantly, 
the 
convexity property allows the length (number
of terms) $N$ to be arbitrarily large while maintaining the same overall bound on the 
coefficient size $\sum_{i=1}^N|\theta_i|\leq\Sc_N$.
This means we can allow the number of terms to grow 
arbitrarily large $N\to\infty$, since $N$ does not appear directly in the results of 
Theorem~\ref{thm:apx:main}. This gives $\epsbase\to0$, while maintaining the same bound $\Sc_N$.
\end{remark}

\section{Application}
\label{sec:apx:application}

In this section, we will first show how he $L_2$ norm bounds of the previous sections can be 
extended to $L_\infty$ (supremum) norm bounds when the target function $f\in\Cs(\Kc)$ is 
differentiable at the origin and Lipschitz on a convex compact set $\Kc\subset\R^n$ including 
the origin. Then we will apply the theory to an approximate extension of the Model Reference 
Adaptive Control (MRAC) setup in \cite{lavretsky2013robust}.

\subsection{Extending the Error Bound to $L_\infty$ for Lipschitz Functions on Convex Compact
Sets}
\label{sec:apx:extLinf}

We assume that $\Kc$ is full dimension, with nonzero Lebesgue measure (volume) 
$\Vc=\mu(\Kc)>0$ and diameter $\Dc>0$. It also must hold that $\Kc\subseteq B_0(\rho)$ for a 
sufficiently large radius $0<\rho\leq\Dc$.

Then, for any approximation $\fhN$ of the form \eqref{eq:apx:baseApx} using an activation
function $\sigma$ satisfying Assumption~\ref{sigAsmpt}, we have
that
\begin{align*}
&\abs{\sum_{i=1}^N \theta_i\,\sigma(\pariX) - \sum_{i=1}^N \theta_i\,\sigma(\pariY)} \\ 
&\leq \sum_{i=1}^N \abs{\theta_i\,\sigma(\pariX) - \theta_i\,\sigma(\pariY)} 
\\
&= \sum_{i=1}^N |\theta_i|\!\abs{\sigma(\pariX) - \sigma(\pariY)}
\leq
\sum_{i=1}^N |\theta_i|\Lsig\abs{\pariX - \pariY}\\
& =  
\sum_{i=1}^N |\theta_i|\Lsig\abs{w_i^\top\!(x-y)} 
\ \leq \  
\sum_{i=1}^N |\theta_i|\!\norm{w_i}_2\,\norm{x-y}_2
\end{align*}
holds for all $x,y\in\R^n$, recalling that $X:=[x_1\ \cdots\ x_n\ 1]^\top$ and $Y:=[y_1\ \cdots\ 
y_n\ 1]^\top$, since $\abs{\sigma(u)-\sigma(v)}\leq\Lsig\abs{u-v}$ for any $u,v\in\R$ by the 
Lipschitz condition of Assumption~\ref{sigAsmpt}. Thus, such approximations are 
$\widehat\Lc$-Lipschitz with constant
\begin{equation}\label{eq:apx:Lhat}
\widehat\Lc \ = \ \Lsig\sum_{i=1}^N |\theta_i|\!\norm{w_i}_2 \ \ .
\end{equation}

A similar calculation gives that such approximations are $\widehat\Ds$-bounded with constant
\begin{equation}\label{eq:apx:Dhat}
\widehat\Ds \ = \ \Dsig\sum_{i=1}^N |\theta_i|
\end{equation}
if $\abs{\sigma(u)-\sigma(v)}\leq\Dsig$ for any $u,v\in\R$ by the bounded condition of 
Assumption~\ref{sigAsmpt}.

Let an affine shift to function $f$ be defined as
\begin{equation}\label{eq:apx:fshift}
\fch(x) \ := \ f(x) - \as^\top x - \bs
\end{equation}
for any $\as\in\R^n$ and $\bs\in\R$.
Thus, if $f$ is $\Lc$-Lipschitz, then
\begin{align*}
&\abs{\fch(x)-\fch(y)} \ = \ 
\abs{f(x)-f(y) - \as^\top(x-y)} \\
&\leq \ 
\abs{f(x)-f(y)} + \abs{\as^\top(x-y)} \\
&\leq\ \Lc\norm{x-y}_2 + \norm{\as}_2\norm{x-y}_2
\end{align*}
holds for all $x,y\in\R^n$. And so, $\fch$ is also Lipschitz, with constant
$\widecheck\Lc = \Lc + \norm{\as}_2$.

\begin{lemma}\label{lem:apx:extLinf}
\textit{
Let $f:\R^n\to\R$ be $\Lc$-Lipschitz on a convex compact set $\Kc\subset\R^n$ containing the 
origin and differentiable at the origin. Let there also be an approximation $\fh:\R^n\to\R$
which is $\widehat\Ds$-bounded or $\widehat{\Lc}$-Lipschitz on $\Kc$ and such that the 
approximation error $\fw(x)=\fch(x)-\fh(x)$, with $\fch$ as in \eqref{eq:apx:fshift}, satisfies 
the $L_2(\Kc,\mubKc)$ norm bound}
$$
\normK{\fw} = \sqrt{\int_\Kc\abs{\fw(x)}^2\mubKcdx} \ \leq \ \epsbase
$$\noindent
\textit{
where the bound $\epsbase>0$ is a finite constant and $\mubKcdx$ is the uniform probability
measure on $\Kc$ with $\int_\Kc\mubKcdx=1$.}

\textit{
Assume $\Kc$ is full dimension with diameter $\Dc>0$, and it holds that $\Kc\subseteq 
B_0(\rho)$ for a ball of sufficiently large radius $0<\rho\leq\Dc$.}
\textit{
Then, the approximation error also satisfies the $L_\infty(\Kc)$ norm bound}
\begin{align}\label{eq:apx:lemDLinfBound}
&\sup_{x\in\Kc}\abs{\fw(x)} \ \leq \ 
2\,
\left(\frac{r}{\Dc(r+\sqrt{r^2+\Dc^2})}\right)^{\!\frac{-n}{n+2}}\,K(\epsbase)
\end{align}
\textit{where}
\begin{align*}
K(\epsbase) \ = \\
i)& \hspace{30pt} \left(\left(\Lc + \norm{\as}_2\right)^n
\epsbase^{\,2}\right)^{\frac{1}{n+2}}+\widehat{\Ds}
\\
ii)& \hspace{30pt}
\left(\left(\Lc + \norm{\as}_2 + \widehat{\Lc}\right)^n
\epsbase^{\,2}\right)^{\frac{1}{n+2}}
\end{align*}
\textit{
and $r$ is the largest radius ball within $\Kc$ centered at its centroid.}
\end{lemma}
\begin{proof}
\if\ARXIV1
Given in Appendix~\ref{app:apx:lemLinfProof}.
\fi
\if\ARXIV0
Given in Appendix III of \cite{lekang2023functionfull}.
\fi
\end{proof}

\subsection{Application to MRAC}
\label{sec:apx:LMRAC}

In Chapter 12 of \cite{lavretsky2013robust}, an approximate extension to the general linear MRAC 
setup (see Chapter 9) is introduced which allows for nonlinearities $f:\R^n\to\R^\ell$ as
$
f(x) \ = \ \Theta^\top\!\Psi(x) + \epsilon_f(x)
$
such that the plant is given by
\begin{align}\nonumber
\dot{x}_t \ &= \ A\,x_t + B\big(u_t + f(x)) \\\label{eq:apx:plant}
&= \ A\,x_t + B\big(u_t + \Theta^\top\!\Psi(x_t) + \epsilon_f(x)\big) \ \ ,
\end{align}
where $A$ is a known $n\times n$ state matrix for the plant state $x_t\in\R^n$, $B$ is a known
$n\times\ell$ input matrix for the input $u_t\in\R^\ell$, and $\Theta$ is an \underline{unknown}
$N\times\ell$ matrix which linearly parameterizes the known vector function $\Psi:\R^n\to\R^N$.
We assume $(A,B)$ is controllable. The general setup also includes an unknown diagonal scaling 
matrix $\Lambda$, such that the overall input matrix is $B\Lambda$, and assumes that $A$ is 
unknown. For simplicity, we assume $A$ is known and omit $\Lambda$.

It is assumed that there exists an $n\times\ell$ matrix of feedback gains $K_x$ and an
$\ell\times\ell$ matrix of feedforward gains $K_r$ satisfying the \textit{matching conditions}
\begin{align*}\nonumber
A + BK_x^\top = A_r \\
BK_r^\top = B_r
\end{align*}
to a controllable, linear reference model
\begin{equation*}
\dot{x}_t^r = A_r\,x_t^r + B_r\,r_t \ \ ,
\end{equation*}
where $A_r$ is a known Hurwitz $n\times n$ reference state matrix for the reference state
$x^r_t\in\R^n$, $B_r$ is a known $n\times\ell$ reference input matrix, and $r_t\in\R^\ell$ is a
bounded reference input. Here, we assume that $K_x$ and $K_r$ can be directly calculated 
from known $A$ and $B$, and used directly in the control law.

It is required that the nonlinearity satisfy the bound
\begin{equation}\label{eq:apx:epsf}
\norm{\epsilon_f(x)}_2 \ \leq \ \bar\varepsilon
\end{equation}
for some constant $\bar\varepsilon>0$ and for all $x\in B_0(r)$ with some radius $r>0$. Then,
the adaptive control law
$$
u_t \ = \ K_x^\top x_t - \Thetah_t^\top\!\Psi(x_t) + \left(1-\mu(x)\right)K_r^\top r_t + 
\mu(x)\uw(x_t)
$$
is shown to stabilize the state tracking error $e_t = x_t-x_t^r$ down to a compact set about the 
origin, where the $N\times\ell$ matrix of parameter estimates $\Thetah_t$ is dynamically updated 
with the update rule
\begin{equation*}
\dot{\Thetah}_t = \Gamma\,\Psi(x_t)\,e_t^\top P_xB \ \ .
\end{equation*}
The scalar function $\mu:\R^n\to[0,1]$
transitions the control law from tracking the reference input to simply returning the plant 
state $x_t$ to the bounded region $B_0(r)$ within which \eqref{eq:apx:epsf} is valid, using an 
appropriately defined $\uw(x_t)$ based on assumptions about the growth of $\epsilon_f(x)$ 
outside of $B_0(r)$. (See Chapter~12 of \cite{lavretsky2013robust} for details.)

And so, if the vector function $\Psi(x)$ is constructed as
\begin{equation}\label{eq:pe:Psi}
\Psi(x)
= 
\begin{bmatrix}
\sigma(\upgamma_1^\top X) \\
\vdots \\
\sigma(\upgamma_N^\top X)
\end{bmatrix} \ ,
\end{equation}
with an activation function $\sigma$ satisfying the bounded (growth) conditions of 
Assumption~\ref{sigAsmpt} with constant $\Dsig$ or $\Lsig$, then the above setup is valid for
any nonlinearity $f = [f_1\ \cdots\ f_\ell]$ where we 
can show that
$$
\sup_{x\in\Kc}\abs{f_i(x)-\Theta_i^\top\Psi(x)} \ \leq \ \bar\varepsilon_i
$$
holds for some constants $\bar\varepsilon_1,\dots,\bar\varepsilon_\ell>0$. Here, 
$f_1,\dots,f_\ell:\R^n\to\R$ and each $\Theta_i^\top\in\R^N$ is the corresponding row of the 
true parameters $\Theta^\top$.

\section{Conclusion and Future Directions}
\label{sec:conclusion}

In this paper, we developed a novel method for bridging approximations using activation 
functions with unknown parameters to approximations using randomly initialized parameters, with 
the mollified integral representation. We showed that this could be extended to the supremum 
norm for use with an extended, approximate version of the linear MRAC setup.

There are two immediate paths that we can see along the lines of extending this work and 
application. 1) Is it possible to do 
``approximate persistency of excitation" on the extended, approximate version of the linear MRAC 
setup, similar to what was achieved in \cite{lekang2022sufficient} but with the parameter 
estimates only converging to a compact set about the origin, instead of being asymptotically 
stable. 2) Extending the work of \cite{irie1988capabilities} towards the determination of 
integral representations for more general classes of functions. In particular, using linear 
combinations of step functions, which cannot be trained using backpropagation in the modern 
paradigm of training deep neural networks. Yet in the 1D case, we can easily show that the step 
function is superior at uniformly approximating functions than the ReLU function. In fact, the 
step can handle discontinuous (regulated) functions, while the ReLU can only deal with 
continuous functions.

\newpage
\printbibliography

\if\ARXIV1
\newpage
\appendices
\onecolumn
\section{}

\subsection{Useful Relations}\label{app:apx:useful}

\subsubsection*{Reverse Triangle Inequality}

Let $\Fc$ be some function space with $f,g\in\Fc$. Let $\|\cdot\|$ be any norm
on $\Fc$. Then we have:
\begin{align*}
\|f\| = \|f - g + g\| \ &\leq \ 
\|f - g\| + \|g\| \\ 
\|g\| = \|g - f + f\| \ &\leq \ 
\|g - f\| + \|f\| = \|f-g\| + \|f\| \ \ ,
\end{align*}
due to the triangle inequality. Therefore, it holds that
\begin{equation}\label{eq:appApx:revTri}
\big|\,\|f\|-\|g\|\,\big| \ \leq \ \|f-g\| \ \ .
\end{equation}\\

\subsubsection*{Norm Difference for Functions with Different Coefficients}

Let $\Fc$ be some function space with $f,g\in\Fc$. Let $\|\cdot\|$ be any norm
on $\Fc$, and $a,b\in\R$. Then we have:
\begin{equation}\label{eq:appApx:normDiffCoef}
\|af-bg\| = \|a(f-g) - \Delta g\|
\leq
|a|\,\|f-g\| + |\Delta|\,\|g\|
\end{equation}
where $\Delta=b-a$, by using the triangle inequality.\\\\

\subsubsection*{Bounding Expected Variation of Sample Average Function from its Mean}

Let $\fb,\fb_1,\dots,\fb_N$ be drawn iid from a function space $\Fc$ and define the
sample average function $\fbbN=\frac{1}{N}\sum_{i=1}^N\fb_i$. With the $L_2(\Bc,\mu)$ norm (for
any valid measure $\mu$) denoted as $\normB{\cdot}$, we have
\begin{align}\nonumber
\E\normB{\fbbN - \E\fbbN}^2 \ &= \ \E\int_\Bc \left(\fbbN - \E\fbbN\right)^2\mu(\dx) \\ \nonumber
&\overset{(a)}{=} \ \int_\Bc \left(\E\left[\fbbN^{\,2}\right]
-2\,\E\left[\fbbN\right]\E\left[\fbbN\right]
+ \E\left[\fbbN\right]^2 \right)\mu(\dx) \\ \nonumber
&\overset{(b)}{=} \ \int_\Bc \left(\E\left[\fbbN^{\,2}\right] -
\E\left[\fbbN\right]^2 \right)\mu(\dx) \ 
= \int_\Bc \text{var}\!\left(\fbbN\right)\mu(\dx) \\ \nonumber
&\overset{(c)}{=} \ \int_\Bc \frac{1}{N}\,\text{var}(\fb)\,\mu(\dx) \ 
= \ \frac{1}{N} \int_\Bc \left(\E\left[\fb^2\right]-\E[\fb]^2\right)\mu(\dx) \\ 
\label{eq:appApx:varBnd}
&= \ \frac{1}{N}\left(\E\normB{\fb}^2-\normB{\E[\fb]}^2\right) \ 
\overset{(d)}{\leq} \ \frac{1}{N}\,\E\normB{\fb}^2 \ \ .
\end{align}
Here (a) is by using Fubini's theorem to bring the expectation inside the integral, the
linearity of expectation, and the outer expectation then being redundant for the last term, (b)
is from combining the middle and last terms, (c) is by the facts of variance of averaged iid
random variables (covariance terms are zero, while means and variances are equal), and (d) is
because the second term is nonpositive and therefore it can be dropped with inequality.\\\\

\subsubsection*{McDiarmid's Inequality}\label{sec:app:McDInq}

Let $\vb_1,\dots,\vb_d$, for any $d\in\N$, be independent random variables in the sets
$\Vc_1,\dots,\Vc_d$ and let the function $h:\Vc_1\times\dots\times\Vc_d\to\R$ satisfy for all
$i\in[d]$ and some positive scalars $k_1,\dots,k_d>0$, that
\begin{equation*}
|h(v_1,\dots,v_i,\dots,v_d)-h(v_1,\dots,\widetilde v_i,\dots,v_d)|\leq k_i
\end{equation*}
holds for all points $(v_1,\dots,v_i,\dots,v_d),(v_1,\dots,\widetilde v_i,\dots,v_d)\in
\Vc_1,\times\dots\times\Vc_d$ differing only at $v_i,\widetilde v_i\in\Vc_i$. Then, for all 
$t\geq0$, it holds that
\begin{equation}\label{eq:appApx:McDIneq}
\P\Big[\,h(\vb_1,\dots,\vb_d)-\E[h(\vb_1,\dots,\vb_d)]\,\geq t\,\Big]
\ \leq \ \exp\left(\frac{-2\,t^2}{\sum_{i=1}^dk_i^2}\right) \ \ .
\end{equation}

\newpage
\subsubsection*{Lemma on Approximating Functions of Convex Combinations}

The following lemma is credited to Maurey in \cite{pisier1981remarques}.

\begin{lemma}\label{lem:appApx:convComb}
If $\fch$ is in the closure of the convex hull of a bounded set of functions $\Fc$ in a Hilbert
space $\Hc$, such that $\norm{f}_{\Hc}\leq b$ holds for all $f\in\Fc$, then for every $N\geq1$,
there exists an $\fbN$ in the convex hull of $N$ points in $\Fc$ such that
\begin{equation*}
\norm{\fch-\fbN}_{\Hc}^2 \ \leq \ \frac{\bar{c}}{N}
\end{equation*}
holds for any $\bar{c} > b^2 - \norm{\fch\,}_{\Hc}^2 \ $.
\end{lemma}
\begin{proof}
Define $\fst$ as a finite but arbitrarily large convex combination of $m$ points
$\fst_i\in\Fc$, meaning it has the form $\fst = \sum_{i=1}^m\gamma_i\fst_i$ with $\gamma_i\geq0$, $\sum_{i=1}^m\gamma_i=1$, and such that
$\norm{\fch-\fst}_{\Hc}\leq\frac{\delta}{\sqrt{N}}$ for some $\delta>0$. Randomly sample
functions $\fb,\fb_1,\dots,\fb_N$ iid from the set $\{\fst_i,\dots,\fst_m\}$ according to
$\P[\fb=\fst_i]=\gamma_i$. Thus, we have $\E[\fbbN]=\E[\frac{1}{N}\sum_{i=1}^N\fb_i]=\E[\fb]
=\fst$, and then using \eqref{eq:appApx:varBnd} gives
$\E\norm{\fbbN-\fst}_{\Hc}^2=\frac{1}{N}(\E\norm{\fb}_{\Hc}^2 - \norm{\fst}_{\Hc}^2) \leq
\frac{1}{N}(b^2 - \norm{\fst}_{\Hc}^2)$. Since the expectation is bounded this way, there must
be a realization $\fbN$ that also satisfies this bound. Then by the triangle inequality, we have
$\norm{\fch-\fbN}_{\Hc} \leq \norm{\fch-\fst}_{\Hc} + \norm{\fbN-\fst}_{\Hc} =
\frac{1}{\sqrt{N}}(\delta + \sqrt{b^2 - \norm{\fst}_{\Hc}^2})$. And since $\fst$ can get
arbitrarily close to $\fch$ by letting $m\to\infty$, and thus $\delta\to0$ and
$\norm{\fst}_{\Hc}\to\norm{\fch\,}_{\Hc}$, this proves the result.
\end{proof}
In our setting, we use the $\normB{\cdot}$ norm for the $L_2(\Bc)$ space. We note that Theorem 1
of \cite{makovoz1996random} gives the same $O(\frac{1}{\sqrt{N}})$ result, but with a superior
constant. The proof there uses the same method as above, but breaks $\Fc$ into $N$ subsets of
diameter $\epsilon$, and then uses that to bound the variance terms. This gives the bound
constant in terms of an inverse covering number of $\Fc$ (denoted $\epsilon_N(\Fc)$), which goes
to zero as $N\to\infty$. However, this requires quantifying $\epsilon_N(\Fc)$, which may be
nontrivial.

\newpage
\section{Proof of Theorem~\ref{thm:apx:moll}}\label{app:apx:mollProof}
First, we observe that \eqref{eq:apx:baseIntRepLam} is equivalently
\begin{align}\nonumber
\fhN(x) \ = \ \sum_{i=1}^{N} \theta_i\,\sigma(\pariX)\,
\int_\Upsilonl \deltalno(\upgamma-\upgamma_i)\,\d\upgamma \ = \ 
\sum_{i=1}^{N} \theta_i\,\int_\Upsilonl \deltalno(\upgamma-\upgamma_i)\,\sigma(\pariX)\,
\d\upgamma \ ,
\end{align}
since $\int_\Upsilonl \deltalno(\upgamma-\upgamma_i)\,\d\upgamma=1$ holds for all $\lambda>0$, 
$n\geq1$, and $\upgamma_i\in\Upsilon$, and $\sigma(\pariX)$ is a constant with respect to
$\d\upgamma$.

Then we define the sets $\Upsilonil$ as the supports of each $\deltalno(\upgamma-\upgamma_i)$,
which is the open cube $(-\frac{1}{\lambda},\frac{1}{\lambda})^{n+1}$ centered at the 
corresponding $\upgamma_i\in\Upsilon$, and note that these are always contained within the
$\lambda$-expanded $\Upsilonl$. Therefore, for any $x\in\Bc$ we have
\begin{align}\nonumber
\big|\fhN(x) - \fhlN(x)\big| \ &=  \ 
\left| \, \sum_{i=1}^{N} \theta_i\,\int_\Upsilonl\deltalno(\upgamma-\upgamma_i)
\,\Big(\sigma(\pariX) - \sigma(\parX)\Big)\,\d\upgamma \, \right| 
\\\nonumber
&\overset{(a)}{\leq}\ \sum_{i=1}^{N} \, |\theta_i| \, \left| 
\int_\Upsilonl\deltalno(\upgamma-\upgamma_i)
\,\Big(\sigma(\pariX) - \sigma(\parX)\Big)\,\d\upgamma \, \right| 
\\\nonumber
&\overset{(b)}{\leq}\ \sum_{i=1}^{N} \, |\theta_i| \, 
\int_\Upsilonl\deltalno(\upgamma-\upgamma_i)
\,\left|\sigma(\pariX) - \sigma(\parX)\right|\,\d\upgamma \\\nonumber
&\overset{(c)}{\leq}\ \sum_{i=1}^{N} |\theta_i| \int_\Upsilonl\deltalno(\upgamma-\upgamma_i)
\sup_{\phi\in\Upsilonil}\left|\sigma(\pariX) - \sigma(\phi^\top X)\right|\d\upgamma
\\\nonumber
&\overset{(d)}{=}\ \sum_{i=1}^{N} |\theta_i| \sup_{\phi\in\Upsilonil}\left|\sigma(\pariX) -
\sigma(\phi^\top X)\right| \int_\Upsilonl\deltalno(\upgamma-\upgamma_i)\d\upgamma \\ \nonumber
&\overset{(e)}{=}\ \Sc_N \sup_{\phi\in\Upsilonil}\left|\sigma(\pariX) -
\sigma(\phi^\top X)\right|
\ \ .
\end{align}
Here (a) is by the triangle inequality, (b) is because $\left|\int
f(x)\,\d x\right|\leq\int|f(x)|\,\d x$ and $\deltalno$ is positive, (c) is because the
integrand is only nonzero on $\upgamma\in\Upsilonil$ by the support of $\deltalno$, (d) is 
because the sup is a constant and can thus be pulled out of the integral, and (e) is by the 
definition of the mollified delta such that 
$\int_\Upsilonl\deltalno(\upgamma-\upgamma_i)\d\upgamma=1$ for all $i\in[N]$ and the (assumed 
known) bound $\sum_{i=1}^N|\theta_i|\leq \Sc_N$.

If $\sigma$ satisfies the overall boundedness assumption with constant $\Dsig$, then we have
\begin{align}\nonumber
\big|\fhN(x) - \fhlN(x)\big| \ 
= \ \Sc_N \, \Dsig
\end{align}
and thus
\begin{align*}
\normB{\fhN(x)-\fhlN(x)} \ &\leq \ 
\sqrt{\int_\Bc \abs{\Sc_N \,\Dsig}^2\,\muBdx} \ = \ \Sc_N \,\Dsig
\end{align*}
where $\Sc_N$ and $\Dsig$ are positive constants.

Otherwise, $\sigma$ satisfies the Lipschitz assumption with constant $\Lsig$ and so we have
\begin{align}\nonumber
\big|\fhN(x) - \fhlN(x)\big| \ 
&\leq \ \Sc_N \, \Lsig
\sup_{\phi\in\Upsilonil}\left|(\upgamma_i -\phi)^\top X\, \right| \\ 
\label{eq:appApx:absDifffNflN}
&\overset{(f)}{\leq} \ \Sc_N \, \frac{\Lsig\,\sqrt{n+1}}{\lambda}\,\|X\|_2
\end{align}
where (f) uses the Cauchy-Schwarz inequality, combined with the fact that
$\|\upgamma_i-\phi\|_2\leq\frac{\sqrt{n+1}}{\lambda}$ for all $\phi\in\Upsilonil$ at any 
$\upgamma_i\in\Upsilonl$. And thus
\begin{align*}
\normB{\fhN(x)-\fhlN(x)} \ &\leq \
\sqrt{\int_\Bc \abs{\Sc_N \,\frac{\Lsig\,\sqrt{n+1}}{\lambda}\,\|X\|_2}^2\,\muBdx} \ = \ 
\Sc_N\ \frac{\Lsig\,\sqrt{n+1}}{\lambda}\ \EB
\end{align*}
where $\Sc_N$, $\Lsig$, $\sqrt{n+1}$, and $\lambda$ are all positive constants.
\vphantom{a}\hfill$\blacksquare$

\newpage

\section{Proof of Lemma~\ref{lem:apx:exp}}\label{app:apx:exp}
The mollified approximation $\fhlN$ has the integral approximation \eqref{eq:apx:mollApx}, which
can be equivalently written as
$$
\fhlN(x) := \int_\Upsilonl \gl(\upgamma_i)\ \sigma(\upgamma^\top
z)\,\d\upgamma
$$
where coefficient function $\gl$ is defined in \eqref{eq:apx:gLam}. Then with $\Pl$ being the
known nonzero density over $\Upsilonl$ used for sampling the parameters, we can equivalently 
write this as
\begin{equation}\label{eq:apx:intRep2}
\fhlN(x) = \int_\Upsilonl \frac{\gl(\upgamma)}{\Pl(\upgamma)}\,\sigma(\upgamma^\top z)
\,\Pl(\upgamma)\d\upgamma
= \Exp_{\upgammab\sim\Pl}\limits
\left[\frac{\gl(\upgammab)}{\Pl(\upgammab)}\,\sigma(\upgammab^\top z)\right]
\ \ .
\end{equation}
For any sampled parametervalues $\upgammab_1,\dots,\upgammab_R$ from $\Upsilonl$, let us denote
$\ftr^*$ as the random approximation that has coefficient values $\vartheta^*\in\R^R$ of
\begin{equation}\label{eq:appApx:cjOpt}
\vartheta_j^*\ =\ \frac{\gl(\upgammab_j)}{\Pl(\upgammab_j)\,R} \ 
=\ \frac{\sum_{i=1}^N\theta_i\,\deltalno(\upgammab_j-\upgamma_i)}{\Pl(\upgammab_j)\,R}
\hspace{40pt}\forall j\in[R] \ \ .
\end{equation}
Let us denote $\d\Pl(\upgamma_j):=\Pl(\upgamma_j)\,\d\upgamma_j$ for all $j\in[R]$, so that in 
the below integrals $\upgamma,\upgamma_1,\dots,\upgamma_R$ all represent separate variables of
integration. Then with the iid assumption on all $\upgammab_j\sim\Pl$, for this $\ftr^*$ we have
\begin{align*}
\Exp_{\upgammab_j\sim\Pl}\limits&\left[\ftr^*(x\,;\upgammab_1,\dots,\upgammab_R)\right]\  = \ 
\Exp_{\upgammab_j\sim\Pl}\limits\left[\sum_{j=1}^{R}\vartheta_j^*\,\sigma(\upgammab_j^\top 
z)\right] \\
&= \ \int_\Upsilonl\cdots\int_\Upsilonl\ \sum_{j=1}^R\frac{\gl(\upgamma_j)}{\Pl(\upgamma_j)\,R}\,
\sigma(\upgamma_j^\top z)\ \d\Pl(\upgamma_1)\,\cdots\,\d\Pl(\upgamma_R) \\
&\overset{(a)}{=}\ \int_\Upsilonl\cdots\int_\Upsilonl \frac{\gl(\upgamma_1)}{\Pl(\upgamma_1)\,R}\,
\sigma(\upgamma_1^\top z)\,\d\Pl(\upgamma_1)\,\cdots\,\d\Pl(\upgamma_R) \ + \ \cdots \\
& \hspace{100pt} + \ \int_\Upsilonl\cdots\int_\Upsilonl \frac{\gl(\upgamma_R)}{\Pl(\upgamma_R)\,R}\,
\sigma(\upgamma_R^\top z)\,\d\Pl(\upgamma_1)\,\cdots\,\d\Pl(\upgamma_R) \\
&\overset{(b)}{=}\ \int_\Upsilonl\frac{\gl(\upgamma_1)}{\Pl(\upgamma_1)\,R}\,
\sigma(\upgamma_1^\top z)\,\d\Pl(\upgamma_1)\  \big(1^{R-1}) \ + \ \cdots \\
& \hspace{150pt} + \int_\Upsilonl\frac{\gl(\upgamma_R)}{\Pl(\upgamma_R)\,R}\,
\sigma(\upgamma_R^\top z)\,\d\Pl(\upgamma_R)\ \big(1^{R-1}) \\
&\overset{(c)}{=}\ R \left(\frac{1}{R}\int_\Upsilonl \frac{\gl(\upgamma)}{\Pl(\upgamma)}\,
\sigma(\upgamma^\top z)\, \d\Pl(\upgamma) \right) \ \ 
\overset{(d)}{=} \ \fhlN(x) 
\ \ .
\end{align*}
Here, (a) is by the linearity of integration, (b) is each $\upgamma_j$ only belongs to the
$\d\Pl(\upgamma_j)$ integral, while the remaining $R-1$ integrals in each term are
$\int_{\Upsilonl}\d\Pl(\upgamma_j)=1$, (c) is because there are $R$ copies of the same integral, and
(d) is \eqref{eq:apx:intRep2}.
\vphantom{a}\hfill$\blacksquare$

\newpage
\section{Proof of Theorem~\ref{thm:apx:main}}\label{app:apx:mainProof}

Since $\fhR$ is assumed to satisfy Lemma~\ref{lem:apx:exp}, then by \eqref{eq:appApx:cjOpt}
it holds for all $j\in[R]$ that
\begin{equation}\label{eq:appApx:cMax}
|\vartheta_j|\ \leq\ \left|\frac{\gl(\upgammab_j)}{\Pl(\upgammab_j)\,R}\right| \ \leq\
\frac{\Cgl}{\pmin\,R} \ \ ,
\end{equation}
where $\Cgl$ is defined in \eqref{eq:apx:Cgl} and $\pmin:=\min\{\Pl(\upgamma)\ |\ 
\upgamma\in\Upsilonl\}$.

From the definition of $h_\Bc$ in \eqref{eq:apx:hB}, and for all points differing only at
\mbox{coordinate $j$} as 
$(\upgamma_1,\dots,\upgamma_j,\dots,\upgamma_R),(\upgamma_1,\dots,\upgammaw_j,\dots,
\upgamma_R)\in\Upsilon^R_\lambda$, we have for all $j\in[R]$ that
\begin{align}\nonumber
\big|h_\Bc(&\upgamma_1,\dots,\upgamma_j,\dots,\upgamma_R)- 
h_\Bc(\upgamma_1,\dots,\upgammaw_j,\dots,\upgamma_R)\big| \\ \nonumber
&=\ \Big|\,\normB{\fhlN(x) - \fhR(x\,;\upgamma_1,\dots,\upgamma_j,\dots,\upgamma_R)} 
- \normB{\fhlN(x) - \fhR(x\,;\upgamma_1,\dots,\upgammaw_j,\dots,\upgamma_R)}\,\Big|\\ 
\nonumber
&\overset{(a)}{\leq}\ \big\|\fhlN(x) - \fhR(x\,;\upgamma_1,\dots,\upgamma_j,\dots,\upgamma_R) 
- \fhlN(x) + \fhR(x\,;\upgamma_1,\dots,\upgammaw_j,\dots,\upgamma_R) \big\|_B 
\\\nonumber
&\overset{(b)}{=}\ \normB{\frac{\gl(\upgammaw_j)}{\Pl(\upgammaw_j)\,R}\,
\sigma(\upgammaw_j^\top X) - \frac{\gl(\upgamma_j)}{\Pl(\upgamma_j)\,R}\,
\sigma(\parjX) } \\[5pt] \nonumber
&\overset{(c)}{\leq}\ \abs{\frac{\gl(\upgammaw_j)}{\Pl(\upgammaw_j)\,R}}
\normB{\sigma(\upgammaw_j^\top X) - \sigma(\parjX)} + 
\bigg|\frac{\gl(\upgamma_j)}{\Pl(\upgamma_j)\,R}-\frac{\gl(\upgammaw_j)}{\Pl(\upgammaw_j)\,R}\bigg|
\normB{\sigma(\parjX)} \\[5pt] \nonumber
&\overset{(d)}{\leq}\ \frac{\Cgl}{\pmin\,R} \bigg(\normB{\sigmach(\upgammaw_j^\top X) -
\sigmach(\parjX)} + 2 \normB{\sigmach(\parjX)} + 2\abs{\sigma(0)} \bigg) 
\end{align}
Here, (a) is by the reverse triangle inequality
\eqref{eq:appApx:revTri}, (b) is by the definition of $\fhR$ such that only $\vartheta_j$ 
differs at
$\upgamma_j\neq\upgammaw_j$, (c) is by the inequality \eqref{eq:appApx:normDiffCoef}, (d) is by
noting that the $\theta_i$ of $\gl$ can be negative (thus the $2$), the absolute value bound
\eqref{eq:appApx:cMax}, and the definition of the shifted activation $\sigmach(u)= 
\sigma(u)-\sigma(0)$ for all $u\in\R$.

If $\sigma$ satisfies the overall boundedness condition of Assumption~\ref{sigAsmpt} with 
constant $\Dsig$, note that $|\sigmach(u)-\sigmach(v)|\leq\Dsig$ holds for all $u,v\in\Ic$
and thus $|\sigmach(u)-\sigmach(0)|=|\sigmach(u)|\leq\Dsig$ also holds for all $u\in\Ic$. And 
so 
we have
\begin{align}\nonumber
\big|h_\Bc(&\upgamma_1,\dots,\upgamma_j,\dots,\upgamma_R)- 
h_\Bc(\upgamma_1,\dots,\upgammaw_j,\dots,\upgamma_R)\big| \
\leq\ \frac{\Dsig\Cgl}{\pmin\,R}\bigg(3 + 2\frac{|\sigma(0)|}{\Dsig}\bigg) \ \ .
\end{align}

Otherwise, $\sigma$ satisfies the Lipschitz condition of Assumption~\ref{sigAsmpt} with 
constant 
$\Lsig$ and so $|\sigmach(u)-\sigmach(v)|\leq\Lsig|u-v|$ holds for all $u,v\in\Ic$
and thus $|\sigmach(u)-\sigmach(0)|=|\sigmach(u)|\leq\Lsig|u|$ also holds for all $u\in\Ic$.  
And so we have
\begin{align}\nonumber
\big|h_\Bc(\upgamma_1,\dots,\upgamma_j,\dots,\upgamma_R)- 
&h_\Bc(\upgamma_1,\dots,\upgammaw_j,\dots,\upgamma_R)\big| \\ \nonumber
&\leq\ \frac{\Lsig\Cgl}{\pmin\,R}\bigg(\normB{(\upgammaw_j - \upgamma_j)^\top X }
+ 2\normB{\parjX} + 2\frac{|\sigma(0)|}{\Lsig}\bigg) \\ \nonumber
&\overset{(e)}{\leq}\ \frac{\Lsig\Cgl}{\pmin\,R}\,\bigg((\Dt+2\upmax)\,\EB + 
2\frac{|\sigma(0)|}{\Lsig}\bigg)
\end{align}
where (e) is by Cauchy-Schwarz and the definitions of $\Dt$, $\upmax$, and $\EB$.

Next, recall from Lemma~\ref{lem:apx:exp} that in order for 
$
\Exp_{\upgammab_j\sim\Pl}\limits\left[\fhR(x\,;\upgammab_1,\dots,\upgammab_R)\right] = \fhlN(x)
$
to hold, it must be that
$$
\fhR(x\,;\upgammab_1,\dots,\upgammab_R) \ = \
\frac{1}{R}\sum_{j=1}^R\,\frac{\gl(\upgammab_j)}{\Pl(\upgammab_j)}\sigma(\upgammab_j^\top z)\ \ 
.
$$
We define a class of functions $\fs:\R^n\to\R$ using the activation function $\sigma:\R\to\R$, 
over the values of $\upgamma\in\Upsilonl$ as
\begin{equation*}
\Fc \ := \ \Big\{\fs(x\,;\upgamma)= \frac{\gl(\upgamma)}{\Pl(\upgamma)}\,
\sigma(\parjX) \ \Big| \ \upgamma\in\Upsilonl \Big\} \ \ .
\end{equation*}
Thus, it holds that 
$$
\fhR(x\,;\upgammab_1,\dots,\upgammab_R) \ = \ \frac{1}{R}\sum_{j=1}^R\,\fs(x\,;\upgammab_j)
$$
with $\upgammab_1,\dots,\upgammab_R$ sampled iid according to $\Pl$, and so the inequality
\eqref{eq:appApx:varBnd} applies.

And so, ith $h_\Bc$ as defined in \eqref{eq:apx:hB} we have
\begin{align}\nonumber
\Exp_{\upgammab_j\sim\Pl}\limits\left[h_\Bc(\upgammab_1,\dots,\upgammab_R)^2\,\right] 
\ &= \ \Exp_{\upgammab_j\sim\Pl}\limits\normB{\,\fhlN(x) - 
\fhR(x\,;\upgammab_1,\dots,\upgammab_R)}^2
\\\nonumber
&=\ \Exp_{\upgammab_j\sim\Pl}\limits \normB{\,\fhR(x\,;\upgammab_1,\dots,\upgammab_R)
-\Exp_{\upgammab_j\sim\Pl}\limits\left[\fhR(x\,;\upgammab_1,\dots,\upgammab_R)\right]}^2 \\ 
\nonumber
&\overset{(a)}{\leq}\ \frac{1}{R}\Exp_{\upgammab\sim\Pl}\limits
\normB{\frac{\gl(\upgammab)}{\Pl(\upgammab)}\sigma(\parbjX)}^2 \\ \label{eq:appApx:expHB1}
&\leq\ \frac{\Cgl^2}{\pmin^2\,R}\,\Exp_{\upgammab\sim\Pl}\limits\normB{\sigmach(\parbjX)
+ \sigma(0)}^2 \ \ .
\end{align}
Here, (a) is the inequality \eqref{eq:appApx:varBnd}.  

Note that $\sqrt{\Exp_{\upgammab\sim\Pl}\normB{f(x\,;\upgammab)}^2\,}$ defines a norm over such
randomly parameterized functions. Using the square root of \eqref{eq:appApx:expHB1} then gives
\begin{align*}\nonumber
\Exp_{\upgammab_j\sim\Pl}\limits\left[h_\Bc(\upgammab_1,\dots,\upgammab_R)\right] \ &= \ 
\Exp_{\upgammab_j\sim\Pl}\limits\left[\sqrt{h_\Bc(\upgammab_1,\dots,\upgammab_R)^2\,}\,\right] 
\\
\nonumber
&\overset{(a)}{\leq}\ 
\sqrt{\Exp_{\upgammab_j\sim\Pl}\limits\left[h_\Bc(\upgammab_1,\dots,\upgammab_R)^2\,\right]}
\\\nonumber
&\leq\ \frac{\Cgl}{\pmin\,\sqrt{R}\,}\,\sqrt{\Exp_{\upgammab\sim\Pl}\limits
\normB{\sigmach(\parbjX) + \sigma(0)}^2} \\ \nonumber
&\overset{(b)}{\leq}\ \frac{\Cgl}{\pmin\,\sqrt{R}\,}
\left(\sqrt{\Exp_{\upgammab\sim\Pl}\limits
\normB{\sigmach(\parbjX)}^2\,} + |\sigma(0)|\right) \ \ .
\end{align*}
Here, (a) is Jensen's inequality and (b) is the triangle inequality.

If $\sigma$ satisfies the overall boundedness condition of Assumption~\ref{sigAsmpt} with 
constant $\Dsig$,
\begin{align*}\nonumber
\Exp_{\upgammab_j\sim\Pl}\limits\left[h_\Bc(\upgammab_1,\dots,\upgammab_R)\right] 
\ \leq \ 
\frac{\Dsig\,\Cgl}{\pmin\,\sqrt{R}\,}\,
\left(1 + \frac{|\sigma(0)|}{\Dsig}\right)
\end{align*}
then holds. Otherwise, $\sigma$ satisfies the Lipschitz condition with constant $\Lsig$,
\begin{align*}\nonumber
\Exp_{\upgammab_j\sim\Pl}\limits\left[h_\Bc(\upgammab_1,\dots,\upgammab_R)\right] \  
&\leq\ \frac{\Cgl}{\pmin\,\sqrt{R}\,}\left(\Lsig 
\sqrt{\Exp_{\upgammab\sim\Pl}\limits
\normB{\parbjX}^2\,} + |\sigma(0)|\right) \\ \nonumber
&\overset{(c)}{\leq}\ \frac{\Lsig\,\Cgl}{\pmin\,\sqrt{R}\,}\,
\left(\upmax\EB + \frac{|\sigma(0)|}{\Lsig}\right)
\end{align*}
then holds, where (c) is Cauchy-Schwarz and the definitions of $\upmax$ and $\EB$.

Applying McDiarmid's inequality (see Appendix~\ref{app:apx:useful}) with $h_\Bc$ gives
$$
\P\left[\,\Big|h_\Bc(\upgammab_1,\dots,\upgammab_R)-\Exp\big[h_\Bc(\upgammab_1,\dots,\upgammab_R)\big]
\Big|\geq t \right] \ \leq \ 
\exp\left(\frac{-2\,t^2}{\sum_{j=1}^R k_j^2}\right) \ \ .
$$
Above we have shown that the $k_j$ bounds for all $j\in[R]$ are
$$
k_j \ = \
\frac{\Cgl}{\pmin\,R}\,A
$$
and the expected value bound is
$$
\Exp\big[h_\Bc(\upgammab_1,\dots,\upgammab_R)\big] \ \leq \ 
\frac{\Cgl}{\pmin\,\sqrt{R}\,}\,B
$$
for coefficients $A$ and $B$ that depend on properties of the activation function $\sigma$ and 
the sets $\Bc$ and $\Upsilonl$.
Therefore, we have for any $t>0$, that
\begin{align}\nonumber
\P\Bigg[\,\Bigg|h_\Bc(\upgammab_1,\dots,\upgammab_R) &-
\frac{\Cgl}{\pmin\,\sqrt{R}\,}\,B\bigg|
\geq t \Bigg] \\ \nonumber
&\leq \
\P\Big[\,\big|h_\Bc(\upgammab_1,\dots,\upgammab_R)-\Exp\big[h_\Bc(\upgammab_1,\dots,\upgammab_R)\big]
\big|\geq t\Big]  \\ \label{eq:appApx:McDB}
&\leq\ \exp\left(\frac{-2\,t^2}{R\left(\frac{\Cgl}{\pmin\,R}\,A\right)^2} \right) \ = 
\ \exp\left(\frac{-2\,\pmin^2\,R\,t^2}{\Cgl^2\,A^2} \right) \ .
\end{align}
Setting the RHS of \eqref{eq:appApx:McDB} equal to $\nu\in(0,1)$ and solving for $t$ then gives
\begin{equation}\label{eq:appApx:tVal}
\log\left(\frac{1}{\nu}\right) = \frac{2\,\pmin^2\,R\,t^2}{\Cgl^2\,A^2}
\quad\implies\quad
t \ = \ \frac{\Cgl\,A}{\pmin\,\sqrt{R}}\,
\sqrt{\frac{1}{2}\log\left(\frac{1}{\nu}\right)} \ \ .
\end{equation}
Combining \eqref{eq:appApx:McDB} and \eqref{eq:appApx:tVal} gives the result, for any
$\nu\in(0,1)$, as the complement of
\begin{align*}
&\P\Bigg[\,h_\Bc(\upgammab_1,\dots,\upgammab_R) \ \geq \ 
\frac{\Cgl}{\pmin\,\sqrt{R}\,}\,B \ + \ 
\frac{\Cgl\,A}{\pmin\,\sqrt{R}}\,\sqrt{\frac{1}{2}\log\left(\frac{1}{\nu}\right)}\ \Bigg]
\\
&=\ \P\Bigg[\,h_\Bc(\upgammab_1,\dots,\upgammab_R) \ \geq \ 
\frac{\Cgl}{\pmin\,\sqrt{R}\,}\,\left(B +
A\,\sqrt{\frac{1}{2}\log\left(\frac{1}{\nu}\right)}\ \right)\Bigg] \\
i)\hspace{10pt}
&=\ \P\Bigg[\,h_\Bc(\upgammab_1,\dots,\upgammab_R) \ \geq \ 
\frac{\Dsig\,\Cgl}{\pmin\,\sqrt{R}\,}\,\KD(\nu)\Bigg] \ \leq \ \nu \\
ii)\hspace{10pt}
&=\ \P\Bigg[\,h_\Bc(\upgammab_1,\dots,\upgammab_R) \ \geq \ 
\frac{\Lsig\,\Cgl}{\pmin\,\sqrt{R}\,}\,\KL(\nu)\Bigg] \ \leq \ \nu
\end{align*}
where $i)$ and $ii)$ correspond to which boundedness property of Assumption~\ref{sigAsmpt} the 
activation function $\sigma$ satisfies.

And so, we have the coefficients $\KD$ and $\KL$ defined in terms of the selected $\nu$
as:
\begin{align*}
\KD(\nu)\ :=& 
\ \ 1 + \frac{|\sigma(0)|}{\Dsig} +  
\bigg(3 + 2\frac{|\sigma(0)|}{\Dsig}\bigg) \sqrt{\frac{1}{2}\log\left(\frac{1}{\nu}\right)}\\
\KL(\nu)\ :=& 
\ \ \upmax\EB + \frac{|\sigma(0)|}{\Lsig} + \bigg((\Dt+2\upmax)\,\EB + 
2\frac{|\sigma(0)|}{\Lsig}\bigg) \sqrt{\frac{1}{2}\log\left(\frac{1}{\nu}\right)} \ \ .
\end{align*}

Now by applying the triangle inequality, we then have 
\begin{align*}
\normB{\fch(x)-\fhR(x\,;\upgammab_1,\dots,\upgammab_R)} \ \leq \ 
&\normB{\fch(x)-\fhN(x)} + \normB{\fhN(x)-\fhlN(x)} \\
&\hspace{40pt} +\normB{\fhlN(x)-\fhR(x\,;\upgammab_1,\dots,\upgammab_R)} \ .
\end{align*}
The first term is assumed bounded by $\epsbase$ for the base approximation $\fhN$. The second 
term is bounded by Theorem~\ref{thm:apx:moll} and is discussed in Section~\ref{sec:apx:moll}. 
And the third term is bounded above.

We pull out front the activation function constant $\Dsig$ or $\Lsig$ and the base approximation
coefficient size bound $\sum_{i=1}^N|\theta_i|\leq \Sc_N$, since they are common to the last two 
terms. This follows for the last term because
$$
\frac{\Cgl}{\pmin\,R}\ =\ \sum_{i=1}^N|\theta_i|\left(\frac{\eta\lambda}{e}\right)^{n+1}
\frac{1}{\pmin\,R} \ \leq \ \Sc_N\,\frac{\lambda^{n+1}}{R}\ \frac{\eta^{n+1}}{e^{n+1}\,\pmin}\ \ 
,
$$
and so the remaining $\frac{\eta^{n+1}}{e^{n+1}\pmin}$ is added to the existing
coefficient in terms of $\nu$.
\vphantom{a}\hfill$\blacksquare$


\newpage
\section{Proof of Lemma~\ref{lem:apx:extLinf}}\label{app:apx:lemLinfProof}

The Lipschitz constants of $\fch$ and $\fh$ are respectively
$\widecheck\Lc=\Lc+\norm{\nabla\!f(0_n)}_2$ and $\widehat\Lc$. Then, the approximation error
function $\fw=\fch-\fh$ is also Lipschitz, satisfying
\begin{align*}
\abs{\fw(x)-\fw(y)} = \abs{\fch(x)-\fh(x)-\fch(y)+\fh(y)} &\leq 
\abs{\fch(x)-\fch(y)} + \abs{\fh(x)-\fh(y)} \\
&\leq \ (\widecheck\Lc + \widehat\Lc\,)\norm{x-y}_2
\end{align*}
for all $x,y\in\Kc$. And so we define $\Lcw:=\widecheck\Lc + \widehat\Lc$ as the Lipschitz
constant.

We note here that 
$$
\mu\big(B_x(\rho)\big) \ = \ \rho^n\frac{\pi^{\frac{n}{2}}}{\Gamma(\frac{n}{2}+1)}
$$
is the Lebesgue measure (volume) of the $n$-dimensional ball of radius $\rho$ centered at any
$x\in\R^n$.
Let us define for any $h\geq0$ the set
$$
\Omega_h := \left\{x\in\Kc \ \ \Big|\ \ \big|\fw(x)| \ \geq \ h
\, \right\} 
$$
and note that the following therefore holds:
\begin{align*}
\epsbase^2 \geq \int_\Kc|\fw(x)|^2\,\mubKcdx \geq \int_{\Omega_h}|\fw(x)|^2\,\mubKcdx \geq
&\int_{\Omega_h} h^2\,\mubKcdx = h^2\mubKc(\Omega_h)\\
&\implies \ 
\mubKc(\Omega_h) \ \leq \ \frac{\epsbase^2}{h^2} \ \ .
\end{align*}
Now, let there be a $\xh\in\Kc$ satisfying $|\fw(\xh)|= h$ for some finite
$h > \epsbase\,$, and let us define the set
$$
S_d\ := \ \left\{x\in\Kc\ \ \Big| \ \ \norm{\xh-x}_2\leq\frac{h}{d\Lcw}\,\right\}
\ = \ \Kc\medcap B_{\xh}\!\left(\frac{h}{d\Lcw}\right)
$$
for any $d\geq\min\{1,\frac{h}{\Lcw\Dc}\}$, which respectively satisfies
$\norm{\xh-x}_2\leq\frac{h}{\Lcw}$ and $\norm{\xh-x}_2\leq\Dc$. We then get that
\begin{align*}
|\fw(x)| \ = \ |\fw(\xh)-\fw(\xh)+\fw(x)| \ &\geq \ |\fw(\xh)| - |\fw(\xh)-\fw(x)| \\
&\geq \ |\fw(\xh)| - \Lcw\,\norm{\xh-x}_2 \ \geq \ h - \Lcw\,\frac{h}{d\Lcw} \ = \
\frac{d-1}{d}h
\end{align*}
for all $x\in S_d$, by using the reverse triangle inequality.

If $\xh$ is such that $S_d$ is entirely within the interior of $\Kc$, then the intersection is
the entire ball and we thus have
$$
\mubKc(S_d) \ = \ \mubKc\!\left(B_{\xh}\!\left(\frac{h}{d\Lcw}\right)\right) \ = \ 
\frac{1}{\Vc}\left(\frac{h}{d\Lcw}\right)^n\frac{\pi^{\frac{n}{2}}}{\Gamma(\frac{n}{2}+1)} \ \ .
$$
As $\xh$ moves toward a flat surface boundary of $\Kc$, then only half of the ball remains in 
the intersection, which is the limit in the $n=1$ case. But in higher dimensions $n\geq2$, the
compact convex geometry of $\Kc$ means reductions of more than half are possible. 

Figure~\ref{fig:apx:Kgeo} uses an $n=2$ example to illustrate visually of how $\xh$ being
located on a convex curved surface or convex vertex of intersecting lines gives further
reductions compared to the flat surface.
\begin{figure}[H]
\centering
\begin{tikzpicture}
  \def\p{2}
  \def\D{1.2*\p}
  \def\th{90}

  \node[] at (0.95*\D,0) {$\Kc$};

  \fill [green!30] (0,0) --  (\th:\p) arc(\th:-\th:\p) -- cycle;

  \draw[dashed]
  (0,0) node[left] {$\xh$} coordinate (x) -- (0,1.1*\p) coordinate (top);
  \draw[dashed] (0,0) -- (0,-1.1*\p);

  \draw[] (0,0) circle (\p);
  \node[] at (-0.4*\p,0.5*\p) {$B_{\xh}(\frac{h}{d\Lcw})$};
  \draw[fill] (0,0) circle (.05);
  \node[] at (0.5*\p,0) {$S_d$};
\end{tikzpicture}
\hspace{10pt}
\raisebox{60pt}{$\geq$}
\hspace{10pt}
\begin{tikzpicture}
  \def\p{2}
  \def\D{1.2*\p}
  \def\r{1.8}

  \begin{scope}
    \clip(0,0) circle (\p);
    \fill[green!30](\r,0) circle (\r);
  \end{scope}

  \draw[dashed](0,0) arc(180:100:\r);
  \draw[dashed](0,0) arc(180:260:\r);

  \node[] at (0.95*\D,0) {$\Kc$};

  \node[] at (-0.25,0) {$\xh$};

  \draw[opacity=0] (0,0) -- (0,1.1*\p);
  \draw[opacity=0] (0,0) -- (0,-1.1*\p);

  \draw[] (0,0) circle (\p);
  \node[] at (-0.4*\p,0.5*\p) {$B_{\xh}(\frac{h}{d\Lcw})$};
  \draw[fill] (0,0) circle (.05);
  \node[] at (0.5*\p,0) {$S_d$};
\end{tikzpicture}
\hspace{10pt}
\raisebox{60pt}{$\geq$}
\hspace{10pt}
\begin{tikzpicture}
  \def\p{2}
  \def\D{1.2*\p}
  \def\th{45}

  \node[] at (0.95*\D,0) {$\Kc$};

  \fill [green!30] (0,0) --  (\th:\p) arc(\th:-\th:\p) -- cycle;

  \draw[dashed] 
  (0,0) node[left] {$\xh$} coordinate (x) -- ($({1.1*\p*cos(\th)},{1.1*\p*sin(\th)})$);
  \draw[dashed] (0,0) -- ($({1.1*\p*cos(\th)},-{1.1*\p*sin(\th)})$);

  \draw[opacity=0] (0,0) -- (0,1.1*\p);
  \draw[opacity=0] (0,0) -- (0,-1.1*\p);

  \draw[] (0,0) circle (\p);
  \node[] at (-0.4*\p,0.5*\p) {$B_{\xh}(\frac{h}{d\Lcw})$};
  \draw[fill] (0,0) circle (.05);
  \node[] at (0.6*\p,0) {$S_d$};
\end{tikzpicture}
\caption{\label{fig:apx:Kgeo} The intersection $\Kc\medcap 
B_{\xh}\!\left(\frac{h}{d\Lcw}\right)$ in green, for possible convex geometrical features on the 
boundary of $\Kc$ where the center $\xh$ can lie.}
\end{figure}
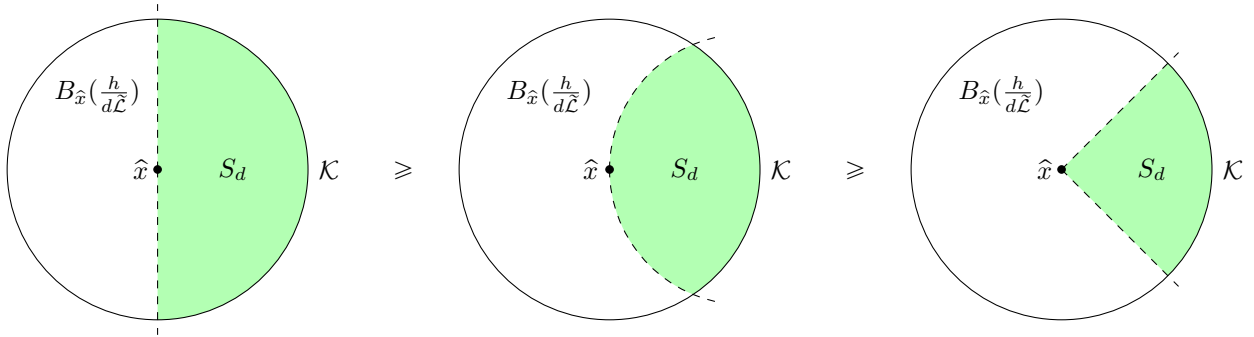

And so, there must exist at least one minimizing boundary location satisfying
$$
\mubKc(S_d)^\star \ := \ 
\min_{x\in\partial\Kc}\mubKc(S_d) \ = \ 
\frac{\ps_\Kc}{\Vc}\left(\frac{h}{d\Lcw}\right)^n\frac{\pi^{\frac{n}{2}}}{\Gamma(\frac{n}{2}+1)}
$$
for some $0<\ps_\Kc<\frac{1}{2}$, since $\Vc>0$ and the boundary of $\Kc$ must contain convex 
curved surfaces and/or vertices.

The exact portion $\ps_\Kc$ of the ball 
contained within $\mubKc(S_d)^\star$ may not be practical to calculate for a given $\Kc$, but we
can always lower bound it in the following manner (based on methods in
\cite{lamperski2021projected} Proof of Lemma 16, Case 1).

Denote $c_\Kc$ as the centroid of $\Kc$ and let $r$ be the radius of the largest ball centered
at $c_\Kc$ contained within $\Kc$. Then define $\ell\geq r$ as the distance from the centroid to 
the closest minimizing boundary location $x^\star$, and align a coordinate system at
$c_\Kc$ such that $x^\star$ is located at position $\big[-\ell\ \ 0\ \cdots\ 0\big]^\top$. Then
$x\in\Kc$ satisfying
$$
-\ell \ \leq x_1 \ \leq \ 0 \qquad \text{ and } \qquad
\sqrt{\sum_{i=2}^{n}x_i^2\,} \ \leq \ r + \frac{r}{\ell\,}\,x_1
$$
defines a conic section which is entirely within $\Kc$. If we then define
$$
\vartheta := \arctan\left(\frac{r}{\ell\,}\right)
\qquad\text{and}\qquad
\dscr + \dscr\sin(\vartheta) = \frac{h}{d\Lcw} \ \ ,
$$
this allows a ball of radius $\dscr\sin(\vartheta)$ to be inscribed within the conic section at
location $\big[-\ell+\dscr\ \ 0\ \cdots\ 0\big]^\top$.

Using the above definitions, we have that this radius satisfies
$$
\dscr\sin(\vartheta) \ = \ \frac{\frac{h}{d\Lcw}}{1+\sin(\vartheta)}\sin(\vartheta) \ = \ 
\frac{\frac{h}{d\Lcw}}{1 + \frac{r}{\sqrt{r^2 + \ell^2}\,}}\,\frac{r}{\sqrt{r^2 + 
\ell^2}\,}
\ = \ \frac{r}{r+\sqrt{r^2+\ell^2}}\,\frac{h}{d\Lcw} \ =: \ \varrho_\Kc\,\frac{h}{d\Lcw}
$$
and since the centroid is within the interior of $\Kc$, the factor $\varrho$ must satisfy
$$
\frac{r}{r+\sqrt{r^2+\Dc^2}\,} \ < \ \varrho_\Kc \ < \ \frac{r}{r+\sqrt{r^2 + r^2}\,} = 
\frac{1}{1+\sqrt{2}} \ \ .
$$
Figure~\ref{fig:apx:centroidK} illustrates an example of this method in $n=2$, with the 
inscribed ball outlined in red.
\begin{figure}[H]
\centering
\begin{tikzpicture}
  \def\r{1.6}
  \def\p{2}
  \def\th{45}
  \def\thC{20}
  \def\c{4.5}
  \def\D{7}
  \def\Dh{1.7}
  \def\DD{8}

  \fill [green!30] (0,0) --  (\th:\p) arc(\th:-\th:\p) -- cycle;

  \draw[dashed] (0,0) node[left] {$x^\star$} coordinate (x) --
  ($({1.1*\p*cos(\th)},{1.1*\p*sin(\th)})$) -- (\D,\Dh); 
  \draw[dashed] (0,0) --
  ($({1.1*\p*cos(\th)},-{1.1*\p*sin(\th)})$) -- (\D,-\Dh); 
  \draw[dashed] (\D,\Dh) arc(90:-90:\Dh);

  \fill [red!20] (0,0) --  (\thC:\p) arc(\thC:-\thC:\p) -- cycle;

  \draw[dotted] (\c,0) circle (\r);

  \draw[] (0,0) circle (\p);
  \node[] at (-0.4*\p,0.5*\p) {$B_{x^\star}(\frac{h}{d\Lcw})$};
  \draw[fill] (0,0) circle (.05);
  \node[] at (1.4,1.0) {$S_d$};

  \draw[] (0,0) -- (\c,\r) -- node[right] {$r$} (\c,0) -- (\c,-\r);
  \draw[] (0,0) -- (\c,-\r);

  \draw[-] (0,0) -- node[below right] {$\ell$} (\c,0);

  \node[] at (0.8,0.15) {$\vartheta$};

  \draw[fill] (\c,0) circle (.05);
  \node[] at (\c+.3,-0.25) {$c_\Kc$};

  \node[] at (7.3,1.1) {$\Kc$};

  \draw [thick,red] ($({\p*sqrt(\r*\r +\c*\c)/(\r+sqrt(\r*\r
    +\c*\c))},0)$) circle ({\p * \r/(\r + sqrt(\r*\r
    +\c*\c))});
  \draw[fill] ($({\p*sqrt(\r*\r +\c*\c)/(\r+sqrt(\r*\r
    +\c*\c))},0)$) node[below] {$\dscr$} circle (.05);
\end{tikzpicture}
\caption{\label{fig:apx:centroidK} The largest ball which always lower bounds $S_d$, centered at 
$\dscr$ with radius $\dscr\sin(\vartheta)$.}
\end{figure}
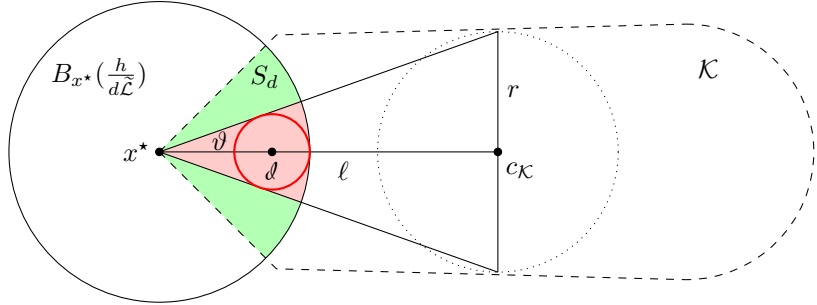

Thus, as $\xh$ moves toward such a minimizing boundary point of $\Kc$, we get the lower bound
$$
\mubKc(S_d) \ \geq \ 
\mubKc(S_d)^\star
\ \geq \ \ 
\frac{\varrho_\Kc^n}{\Vc}\left(\frac{h}{d\Lcw}\right)^{\!n}\frac{\pi^{\frac{n}{2}}}{\Gamma(\frac{n}{2}+1)}
\ =: \ \frac{\rho_{\Kc,n}}{\Vc}\left(\frac{h}{d\Lcw}\right)^{\!n}
\ \ .
$$
On the other hand, there may be $x\in\Kc$ more than $\frac{h}{d\Lcw}$ away from any $\xh$ where
it still holds that $|\fw(x)|\geq\frac{d-1}{d}h$. Thus,
$$
S_d \subseteq\Omega_{\frac{d-1}{d}h}
\quad\implies\quad
\mubKc(S_d) \ \leq \ \mubKc\left(\Omega_{\frac{d-1}{d}h}\right)
\ \leq \ 
\frac{d^2\epsbase^2}{(d-1)^2h^2}
\ \ .
$$
Combining these relations gives that
$$
\frac{\rho_{\Kc,n}}{\Vc}\,\left(\frac{h}{d\Lcw}\right)^n \ \leq \ \mubKc(S_d) \ \leq \ 
\frac{d^2\epsbase^2}{(d-1)^2h^2}
\qquad\implies\qquad
h\ \leq \ \frac{d}{(d-1)^{\frac{2}{n+2}}}\,
\left(\frac{\Vc\Lcw^n\epsbase^2}{\rho_{\Kc,n}}\right)^{\frac{1}{n+2}}
$$
must hold for any $d\geq1$. Then since
$$
\frac{\d}{\d d}\frac{d}{(d-1)^{\frac{2}{n+2}}} \ = \ 
\frac{n(d-1)-2}{(n+2)(d-1)^{\frac{n+4}{n+2}}} = 0
\quad\implies\quad
d \ = \ \frac{n+2}{n} \ \ ,
$$
we have the smallest upper bound on $h$ as
$$
h\ \leq \ \frac{\frac{n+2}{n}}{\frac{2}{n}^{\frac{1}{n+2}}}
\left(\frac{\Vc\Lcw^n\epsbase^2}{\rho_{\Kc,n}}\right)^{\frac{1}{n+2}}
\ \leq \ 2\,(\rho_{\Kc,n})^{\frac{-1}{n+2}}\left(\Vc\Lcw^n\epsbase^2\right)^{\frac{1}{n+2}}
\ \ .
$$
By definition, we have
$$
(\rho_{\Kc,n})^{\frac{-1}{n+2}} \ = \
\left(\varrho_\Kc^n\,\frac{\pi^{\frac{n}{2}}}{\Gamma(\frac{n}{2}+1)}\right)^{\!\frac{-1}{n+2}}
\ = \ 
\left(\frac{r}{r+\sqrt{r^2+\ell^2}}\right)^{\!\frac{-n}{n+2}}
\left(\frac{\pi^{\frac{n}{2}}}{\Gamma(\frac{n}{2}+1)}\right)^{\!\frac{-1}{n+2}}
$$
and since $\Kc\subseteq B_0(\rho)$ for sufficiently large $\rho$, then
$$
\Vc \ \leq \ \rho^n\frac{\pi^{\frac{n}{2}}}{\Gamma(\frac{n}{2}+1)} \ \ .
$$
Thus,
$$
(\rho_{\Kc,n})^{\frac{-1}{n+2}}\,\Vc^{\frac{1}{n+2}} \ \leq \ 
\left(\frac{r}{r+\sqrt{r^2+\Dc^2}}\right)^{\!\frac{-n}{n+2}}\,\Dc^{\frac{n}{n+2}} \ \ .
$$
\vphantom{a}\hfill$\blacksquare$

\fi

\end{document}